\documentclass[11pt]{amsart}
\usepackage{amssymb, amsmath}
\usepackage{amsthm}
\usepackage{amscd}
\usepackage{graphicx}
\newtheorem{theorem}{Theorem}
\newtheorem{lemma}[theorem]{Lemma}

\newtheorem{question}[theorem]{Question}
\newtheorem{corollary}[theorem]{Corollary}
\newtheorem{proposition}[theorem]{Proposition}

\theoremstyle{definition}

\newtheorem*{example}{Example}
\newtheorem*{remark}{Remark}

\title{Ends of finite volume, nonpositively curved manifolds}
\author{Grigori Avramidi}

\def\ra{\rightarrow}

\def\beqa{\begin{eqnarray}}
\def\eeqa{\end{eqnarray}}
\def\beqa{\begin{eqnarray}}
\def\eeqa{\end{eqnarray}}

\DeclareMathOperator{\SL}{SL}
\DeclareMathOperator{\SO}{SO}

\input xy
\xyoption{all}
\begin{document}
\begin{abstract}
We study complete, finite volume $n$-manifolds $M$ of bounded nonpositive sectional curvature. A classical theorem of Gromov says that if such $M$ has {\it negative} curvature then it is homeomorphic to the interior of a compact manifold-with-boundary, and we denote this boundary $\partial M$. If $n\geq 3$, we prove that the universal cover of the boundary $\widetilde{\partial M}$ and also the $\pi_1M$-cover of the boundary $\partial\widetilde M$ have vanishing $(n-2)$-dimensional homology. For $n=4$ the first of these recovers a result of Nguy$\tilde{\hat{\mathrm{e}}}$n Phan saying that each component of the boundary $\partial M$ is aspherical. For any $n\geq 3$, the second of these implies the vanishing of the first group cohomology group with group ring coefficients $H^1(B\pi_1M;\mathbb Z\pi_1M)=0$. A consequence is that $\pi_1M$ is freely indecomposable. These results extend to manifolds $M$ of bounded nonpositive curvature if we assume that $M$ is homeomorphic to the interior of a compact manifold with boundary.
Our approach is a form of ``homological collapse'' for ends of finite volume manifolds of bounded nonpositive curvature. 
This paper is very much influenced by earlier, yet still unpublished work of Nguy$\tilde{\hat{\mathrm{e}}}$n Phan.

\end{abstract}
\maketitle

\section*{Introduction}
A classical theorem of Gromov \cite{gromovnegativelycurved} says that a complete, finite volume Riemannian $n$-manifold of bounded negative curvature ($-1<K<0$) is homeomorphic to the interior of a compact manifold with boundary $\partial M$. \begin{question} 
What can one say about the topology of the boundary $\partial M$? 
\end{question}
If $M$ has pinched negative curvature ($-1<K<-\varepsilon<0$), then $\partial M$ is homeomorphic to an infranil manifold\footnote{In particular, $\partial M$ is aspherical and its fundamental group contains a finite index nilpotent subgroup.} \cite{buserkarcher, eberleinlattices}. However, if the curvature is not pinched there are many more possibilities for $\partial M$   \cite{abreschschroeder,belegradek,fujiwara,nguyenphanfinite,nguyenphansol}, and the general picture is unclear, and one has only a few restrictions on the topology of the boundary (the Euler characteristic, $L^2$-Betti numbers\footnote{Here we mean the $L^2$-Betti numbers of the $\pi_1M$-cover of the boundary.} \cite{cheegergromovbounds} and simplicial volume \cite{gromovfilling} must vanish). Recently, T$\hat{\mathrm{a}}$m Nguy$\tilde{\hat{\mathrm{e}}}$n Phan discovered new constraints on the topology of the boundary in dimension four. In unpublished (forthcoming?) work, she shows that in dimension $n=4$ each component of the boundary $\partial M$ is an aspherical $3$-manifold, and goes on to give more information about which aspherical $3$-manifolds do and which do not occur as boundaries.  
\begin{theorem}[Nguy$\tilde{\hat{\mathrm{e}}}$n Phan]
\label{tamtheorem}
Let $M$ be a complete, finite volume, Riemannian $4$-manifold of bounded negative curvature $-1<K<0$. Then each connected component of $\partial M$ is aspherical. 
\end{theorem}
Inspired by this result, we find extensions of it, phrased in somewhat more homological form, to manifolds of dimension $n\geq 4$. Our methods also extend to 
bounded nonpositive curvature. In contrast to manifolds of bounded negative curvature, these manifolds are not always tame\footnote{Infinite type graph product examples are constructed in section 5 of \cite{gromovnegativelycurved}.}, so in this case we assume tameness\footnote{A manifold $M$ is called {\it tame} if it is the interior of a compact manifold with boundary.}
 as part of the hypotheses. Denote by $\partial\widetilde M$ the $\pi_1M$-cover of the boundary\footnote{Equivalently, it is the boundary of the universal cover $\widetilde M\ra M$, which justfies the notation $\partial\widetilde M$. 
}. Here is our main result.
\begin{theorem}
\label{maintheorem2}
Let $M$ be a complete, finite volume, Riemannian $n$-manifold, $n>2$, of bounded nonpositive curvature $-1<K\leq 0$. If $M$ is homeomorphic to the interior of a compact manifold with boundary $\partial M$ then 
\begin{equation}
\label{nocoeff}
H_{\geq n-2}(\partial\widetilde M)=0.
\end{equation}
\end{theorem}
The cover $\partial\widetilde M\ra\partial M$ might not be connected or simply connected. However, the same methods also prove a homology vanishing result for the universal cover of the boundary $\widetilde{\partial M}$.
\begin{theorem}
\label{maintheorem3}
With the same hypotheses as in Theorem \ref{maintheorem2}, we have
\begin{equation}
\label{coeff}
H_{\geq n-2}(\widetilde{\partial M})=0.
\end{equation}
\end{theorem}
For $n=4$ this implies each component of the boundary is aspherical, recovering Theorem \ref{tamtheorem}.
\begin{remark}
The reader should not confuse $\partial\widetilde M$ and $\widetilde{\partial M}$.\footnote{The later only appears in section \ref{coversandcoefficients}.} 
\end{remark}
\begin{example}
Finite volume locally symmetric manifolds of non-compact type are prominent examples of manifolds to which Theorem \ref{maintheorem2} applies, but for which much finer information about the end is already known. 
If $M$ is such a locally symmetric manifold then (see e.g. \cite{pettetsouto})
\begin{itemize}
\item
$M$ is the interior of a compact manifold-with-boundary, and 
\item
either $M$ is a closed manifold or there is a positive number $q$ so that the $\pi_1M$-cover of the boundary $\partial\widetilde M$ is homotopy equivalent to an infinite wedge of $(q-1)$-spheres $\vee S^{q-1}$. 
\end{itemize}
A simple example is a cartesian product of $q$ non-compact, complete, finite volume, hyperbolic surfaces. A more interesting example is the quotient of $\SL_{q+1}\mathbb R/\SO(q+1)$ by a finite index torsion free subgroup of $\SL_{q+1}\mathbb Z$. Theorem \ref{maintheorem2} shows that a bit of the structure of the end possessed by these manifolds remains even in the absence of symmetry on the universal cover. 
\end{example}


\subsection*{Applications}
We can use Poincare duality to rephrase (\ref{nocoeff}) in a way that does not depend on the dimension of $M$ and only involves the fundamental group $\pi:=\pi_1M$. Let $\overline M$ be the compact $m$-manifold-with-boundary with interior $M$ and boundary $\partial M$. If $m\geq 3$ then 
$$
0=H_{m-2}(\partial\widetilde{M})\cong H_{m-1}(\widetilde{\overline M},\partial\widetilde{ M})\cong H^1_c(\widetilde{\overline M})\cong H^1(\overline M;\mathbb Z\pi)\cong H^1(B\pi;\mathbb Z\pi).
$$ 
If $N$ is {\it any} closed $n$-manifold with fundamental group $\pi$, then $H_{n-1}(\widetilde N)\cong H^1_c(\widetilde N)\cong H^1(N;\mathbb Z\pi)\cong H^1(B\pi;\mathbb Z\pi)=0$ because we can add cells of dimension $\geq 3$ to turn $N$ into a $B\pi$, and this does not affect the first cohomology group. In summary, we get the following corollary.
\begin{corollary}
\label{H^1}
Suppose $\pi$ is the fundamental group of a tame, finite volume, complete Riemannian manifold of bounded nonpositive curvature and of dimension $\geq 3$. Then $H^1(B\pi;\mathbb Z\pi)=0$. Consequently, if $N$ is any closed $n$-manifold with fundamental group $\pi$ then
$H_{n-1}(\widetilde N)=0$.
\end{corollary}
\begin{example}
We can always build a closed $4$-manifold $N^4$ with a given finitely presented fundamental group. If the fundamental group satisfies the hypotheses of the Corollary, then we see that the homology of the universal cover $\widetilde N^4$ of this manifold is concentrated in dimension two. 
\end{example}
\begin{remark}
Several comments about the Corollary are in order.
\begin{itemize}
\item
First, $H_n(\widetilde N)=0$ because $\widetilde N$ is non-compact. 
\item
Taking the connect sum $N\#(S^{n-2}\times S^2)$ changes $H_{n-2}$ of the universal cover but not the fundamental group (for $n\geq 4$), so the Corollary is the best one can hope for. 
\item
The connect sum $N\#(S^{n-1}\times S^1)$ changes $H_{n-1}$ of the universal cover, but it also changes the fundamental group.   
\end{itemize}
\end{remark}
Expanding on the last item, any connect sum $M\#N$ of two closed, non-simply connected $4$-manifolds has $3$-dimensional homology in its universal cover (a nontrivial cycle is represented by the lift of one of the connecting $3$-spheres) so its fundamental group $\pi_1M*\pi_1N$ cannot satisfy the hypotheses of Corollary \ref{H^1}. Since any finitely presented group is the fundamental group of a closed $4$-manifold, the group $\pi$ in Corollary \ref{H^1} cannot be expressed as a free product of two finitely presented groups. In fact, the finite presentation hypothesis is unnecessary\footnote{This was pointed out to me by Mike Davis.}.

\begin{corollary}
\label{indecomposable}
Any group $\pi$ satisfying the hypotheses of Corollary \ref{H^1} is freely indecomposable.\footnote{A group $\pi$ is {\it freely indecomposable} if it cannot be expressed as a free product $A*B$ of nontrivial $A$ and $B$.} 
\end{corollary}
\begin{proof}
Assume $\pi=A*B$ is the free two product of two infinite groups. The $A$-invariant functions on $A*B$ are not compactly supported, so $H^0(BA;\mathbb Z[A*B])=(\mathbb Z[A*B])^A=0$ and similarly $H^0(BB;\mathbb Z[A*B])=0$. So, by the long exact cohomology sequence with $\mathbb Z\pi$-coeffficents\footnote{$H^0(BA;\mathbb Z\pi)\oplus H^0(BB;\mathbb Z\pi)\ra H^0(*;\mathbb Z\pi)\ra H^1(B\pi;\mathbb Z\pi)$}, $H^0(*;\mathbb Z\pi)=\mathbb Z\pi$ injects into $H^1(B\pi;\mathbb Z\pi)$, which contradicts Corollary \ref{H^1}. 
\end{proof}


\subsection*{Methods}
At a heuristic level, the method is to study a cycle $z\subset\widetilde M$ lying over the thin part of the manifold $M$ by viewing it from points at infinity. Sometimes the thick part obstructs the view and one can see only part of the cycle from any given point at infinity. So, one has to look from several different points and assemble the resulting information together via algebraic topology. This is illustrated by the picture at the end of the introduction.  
To figure out where to look from, one takes the isometries that move a point $x$ on the cycle $z$ only a little bit (less that a judiciously chosen $\varepsilon$)
and then finds a point at infinity $\xi$ fixed by all these isometries. Then a neighborhood of $x$ can be seen from the point at infinity $\xi$. In fact, one can pick a neighborhood invariant under these isometries and this leads to topological constraints on the neighborhood that let one fill in portions of the cycle $z$. If the cycle has high enough dimension, then one can assemble this local topological information to show the cycle $z$ bounds over the thin part.  The relevant groups of isometries are nilpotent, and higher rank nilpotent groups give more constraints which let one fill more cycles.


At a more technical level, the main ingredients involved in carrying this out are convexity of distance functions and projections to horospheres (nonpositively curved geometry), virtual nilpotence of groups generated by small isometries (bounded curvature) and a process for assembling topological information about open sets into topological information about their union (algebraic topology). These ingredients are already present in Nguy$\tilde{\hat{\mathrm{e}}}$n Phan's proof of Theorem \ref{tamtheorem}. Once one figures out the topological constraints imposed by a nilpotent group in higher dimensions, one begins to suspect that the higher dimensional analogue of Theorem \ref{tamtheorem} should be (\ref{coeff}) and that a homological proof of it would also give (\ref{nocoeff}). The main difficulty in establishing it is that one has to assemble regions corresponding to nilpotent groups of many different ranks, while in the $4$-dimensional case one can reduce to a problem that only involves assembling regions with abelian groups $\mathbb Z$ and $\mathbb Z^2$. 
For an isometry in a nonabelian nilpotent group, a sublevel set of the isometry is usually {\it not} preserved by the entire nilpotent group\footnote{It is only preserved by elements that commute with the isometry.}, and dealing with the issues that stem from this is one of the main technical points of this paper. One thickens sublevel sets to invariant sublevel sets---this was already done in the $4$-dimensional case, but at most one set had to be thickened then and there was no need to assemble it with anything else---and builds an {\it unfolding space} that unfolds excess overlap caused by the thickening process. Then one assembles everything together in this unfolding space and folds back up at the very end.    
The additional difficulty with extending the result to bounded nonpositive curvature is that one has to deal with groups of small semisimple isometries, while in the $-1<K<0$ case all sufficiently small isometries are parabolic. So, one studies unions of minsets of such semisimple isometries and fills cycles in them.

Below we discuss a special case and then explain how to deal with the difficulties in general. 

\subsection*{An illustrative example}
Suppose $\widetilde M^4$ is the universal cover of a $4$-manifold of bounded negative curvature ($-1<K<0$). Let $U=\cup U_i$ and $V=\cup V_i$ be finite unions of sublevel sets of infinite order, orientation preserving isometries $U_i=\{x\in\widetilde M\mid d(x,\gamma_ix)<\varepsilon\}, V_i=\{x\in\widetilde M\mid d(x,\rho_ix)<\varepsilon\}$ such that the $\gamma_i$ have a common power $\gamma=\gamma_i^{r_i}$ and the $\rho_i$ have a common power $\rho=\rho_i^{s_i}$. Decompose a $2$-cycle $z\subset U\cup V$ as a union $z=a\cup b$ with $a\subset U,b\subset V$ and $\partial a=-\partial b\subset U\cap V$. Notice that $U$ is a union of $\gamma$-invariant convex sets, and similarly $V$ is a union of $\rho$-invariant convex sets. 
Moreover, if $U_i$ and $V_j$ intersect and $\varepsilon$ is picked sufficiently small, then the Margulis lemma implies that $\gamma_i$ and $\rho_j$ generate a virtually nilpotent group. Let us {\it assume this is a free abelian group of rank two}, so that $\gamma_i$ and $\rho_j$ commute. Then $U\cap V$ is a union of convex sets $U_i\cap V_j$ each invariant under the free abelian group $\left<\gamma,\rho\right>$. 

We can find a point at infinity $\xi\in\partial_{\infty}\widetilde M$ such that all the $\gamma_i$ preserve horospheres $H(x,\xi)\cong\mathbb R^3$ centered at $\xi$. Let $p:\widetilde M\ra H(x,\xi)\cong \mathbb R^3$ be the horospherical projection (see subsection \ref{horosphericalprojection}). Then any geodesic ray $[y,\xi)$ starting in $U_i$ and going to $\xi$ remains in $U_i$, so each $U_i$ decomposes as a union of such geodesic rays $\mathbb R\times p(U_i)$ and $U=\cup U_i$ decomposes as $U=\mathbb R\times p(U)$. Since $\gamma$ preserves each of the $p(U_i)$, it acts {\it homotopically trivially} on $p(U)=\cup p(U_i)$. It follows from this that $H_{2}(U)=H_{2}(p(U))=0$.\footnote{The intuition behind this homology vanishing is as follows: If $z$ is a non-trivial homology $2$-cycle in $p(U)$ then it links a pair of points $x,y$ in $\mathbb R^3\setminus p(U)$. A sufficiently large translate $\gamma^Nz$ unlinks the cycle from $\{x,y\}$ and this cannot happen if the action of $\gamma$ on $p(U)$ is homotopically trivial.} A similar argument (for a possibly different point at infinity $\eta$) gives $H_2(V)=0$. On the intersection $U\cap V$ there is a homotopically trivial action of a rank two free abelian group $\left<\gamma,\rho\right>$, and it follows from this that $H_1(U\cap V)=0$.

Putting these ``local'' homology vanishing results together, we conclude that the cycle $z$ bounds in $U\cup V$: First, since $H_1(U\cap V)=0$ we find a $2$-chain $c$ in $U\cap V$ such that $\partial c=\partial a$. Then $z=(a-c)+(c+b)$ where $a-c$ is a cycle in $U$ and $c+b$ is a cycle in $V$.\footnote{This is shown in the left picture below.} 
Both of these bound, since $H_2(U)=0=H_2(V)$, so $z$ bounds inside $U\cup V$. 

\subsection*{The general argument}
The aim of this paper is to make the arguments sketched in this example work more generally. In case of bounded negative curvature this consists of two main steps corresponding to the second and third paragraph of the above example. The first step is a ``local vanishing lemma'' which captures the topology of the end that can be seen from a single point at infinity by projecting to a single horosphere and using the action of a single nilpotent group. 
The second step is a ``Mayer-Vietoris assembly'' argument  which organizes the nilpotent groups and corresponding sublevel sets into {\it clumps} whose 
ranks grow when one takes intersections and uses this to assemble the local information together.

\begin{figure}
\centering
\includegraphics[scale=0.28]{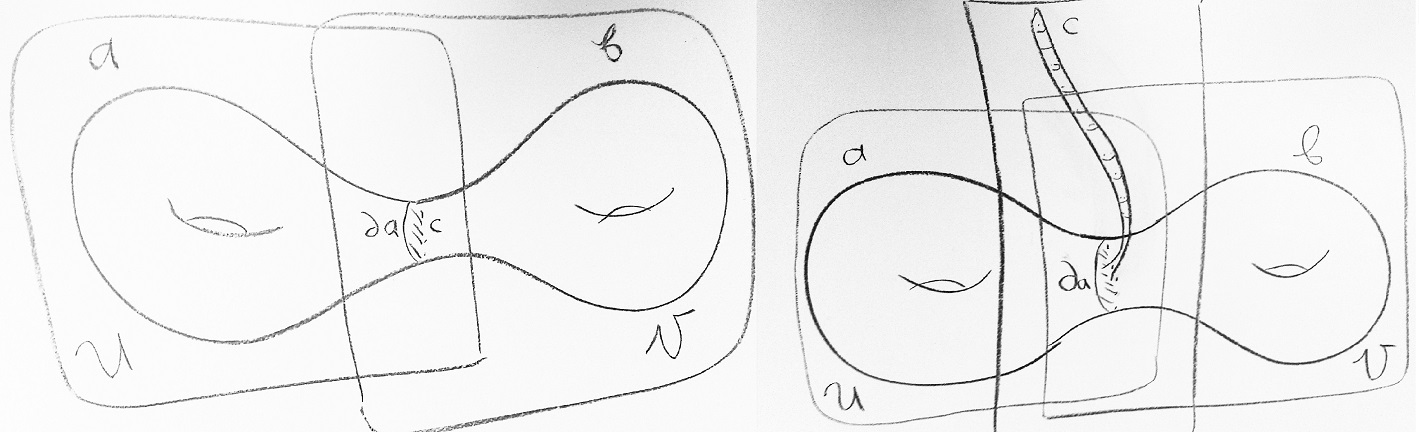}
\end{figure}

The principal difficulty is that, in general, if two sublevel sets $U_i$ and $V_j$ intersect then the group $\left<\gamma_i,\rho_j\right>$ generated by the corresponding isometries is virtually nilpotent (rather than abelian)\footnote{Even if $\varepsilon$ is picked to be very small.} 
and the intersection of sublevel sets $U_i\cap V_j$ is {\it not} invariant under this group. But, such invariance is crucial for the local vanishing portion of the argument. We deal with this problem by thickening the sublevel sets to larger, invariant sublevel sets defined via small central elements in the relevant nilpotent groups (see subsection \ref{patches}). This, in turn, makes the assembly portion of the argument more involved\footnote{Because the cycle $\partial a=-\partial b$ in $U\cap V$ only bounds a chain $c$ in the thickened version of $U\cap V$, and thus the cycles $a-c$ and $c+b$ do not lie in $U$ and $V$, respectively. This is illustrated in the right picture above.}
 and is dealt with by using an ``unfolding 
space construction'' to patch the local vanishing results 
together.

In the case of bounded nonpositive curvature one also needs to fill portions of the cycle $z$ that lie over regions that are ``thin for a semisimple reason''. The methods for dealing with parabolic isometries described above reduce the additional problems to a union of (intersections of) minsets of small semisimple isometries. For each minset there is an abelian group of semisimple isometries and if two minsets intersect then the abelian groups almost commute. A cycle of high enough dimension in such an ``almost abelian arrangement of minsets'' can be filled in a larger arrangement\footnote{We need to use larger minsets corresponding to finite index subgroups because the relevant groups of semisimple isometries only {\it almost} commute. Because of the curvature bound $(-1<K\leq 0)$ the relevant index is bounded by an (extremely large) constant that only depends on the dimension of the manifold $M$. This is why we can make sure that the larger arrangement also lies over the thin part.} which still lies over the thin part. This is proved by a local\footnote{Here, the word ``local'' means a union of minsets acted upon by a single abelian group.} 
vanishing result (Lemma \ref{semisimplevanish}, which boils down to a splitting of minsets of semisimple isometries and the fact that finite order orientation preserving isometries have fixed sets of codimension $\geq 2$) together with an assembly argument (Proposition \ref{almostabelianarrangements}). These two semisimple steps are similar in spirit to but different in detail from their parabolic counterparts. 

\subsection*{A guide to the proof of Theorem \ref{maintheorem2}} 
Some language ($\varepsilon$-patches and clumps) is introduced in section \ref{clumps} and subsections \ref{patches},\ref{small}, and \ref{big}. The special case involving rank $n-1$ nilpotent groups is dealt with in subsections \ref{end} through \ref{rankn-1}. This is the part of the argument where the tameness assumption enters. A very small $\varepsilon$ is picked in subsection \ref{smallepsilon}. With this out of the way, the two main steps of the argument in negative curvature (the local vanishing step and the assembly step) are contained in subsections \ref{parabolicclumps} and \ref{parabolicunfolding}, respectively. The remainder of section \ref{mainsection} deals with nonpositive curvature. Subsections \ref{semisimplereduction1} and \ref{semisimplereduction2} reduce the problem to an almost abelian arrangement of minsets. Subsections \ref{anothersemisimpleassemblylemma} and \ref{semisimpleassembly} show that a cycle of high enough dimension in such an arrangement of minsets bounds inside a larger arrangment of minsets, completing the proof. 

\subsection*{Homology of the universal cover of the boundary} In the example above, the homotopically trivial $\gamma$-action on the union $U=\cup U_i$ in $\partial\widetilde M$ lifts to a homotopically trivial $\gamma$-action on the inverse image of this union in $\widetilde{\partial M}$, and this leads to a local vanishing result for the homology of $\widetilde{\partial M}$. To prove Theorem \ref{maintheorem3} one does this more generally and checks that the assembly steps and semisimple reduction steps work the same way as in the proof of Theorem \ref{maintheorem2}.
%



\subsection*{Local collapse and nilpotent structures}
The idea that a space of bounded geometry has a thick-thin decomposition in which the thin part locally collapses to a lower dimensional space is a general theme (see \cite{cheegerfukayagromov,cheegergromovcollapse1}, section 5 of 
\cite{gromovicm}, and section 6 of \cite{gromovasymptotic}). Usually, some sort of local collapse structure or nilpotent structure is involved that encodes how the different nilpotent groups (obtained via the Margulis lemma in bounded geometry) fit together. The existence of such a collapse structure is known to have topological consequences, such as vanishing of the Euler characteristic, simplicial volume and $L^2$-Betti numbers. These are all consequences of the bounded geometry. In non-positive curvature, local collapse structures have also proved useful in work of Farrell and Jones on topological rigidity without compactness assumptions \cite{farrelljones,farrelljoneslocalcollapse}. From one point of view, Nguy$\tilde{\hat{\mathrm{e}}}$n Phan's result shows how negative curvature lets one extract additional global topological information out of the local collapse structure via horospherical projections. The present paper elaborates on this theme.  

\subsection*{A lack of examples}
There is a very wide gap between the constraints given by Theorem \ref{maintheorem2} and the examples of ends of manifolds with bounded nonpositve curvature constructed thus far (see \cite{abreschschroeder,fujiwara} and the survey \cite{belegradeksurvey}), especially in dimensions $n>4$. A promising tool for narrowing this gap might be the work of Ontaneda on smooth Riemannian hyperbolization, and especially its relative versions \cite{ontaneda}. So far, this method has been used to realize many almost flat manifolds  (Theorem A of \cite{ontaneda}), closed $3$-manifolds with Sol geometry \cite{nguyenphansol}, and products $S^1\times M$ where $M$ is a closed nonpositively curved manifold \cite{belegradek} as ends of complete, finite volume manifolds with bounded negative curvature. 

\begin{remark}
The Riemannian hyperbolization process naturally produces ends of manifolds with bounded negative curvature ($-1<K<0$) rather than bounded nonpositive curvature ($-1<K\leq 0$). However, it is not clear whether there is any difference between these two notions as far as ends are concerned\footnote{For instance, the author does not know whether the end of a Cartesian product of $n\geq 2$ punctured tori $(\dot{\mathbb T}^2)^n$ is also the end of a complete, finite volume, $2n$-manifold of bounded negative curvature.}.
\end{remark}

\begin{figure}
\centering
\includegraphics[scale=0.27]{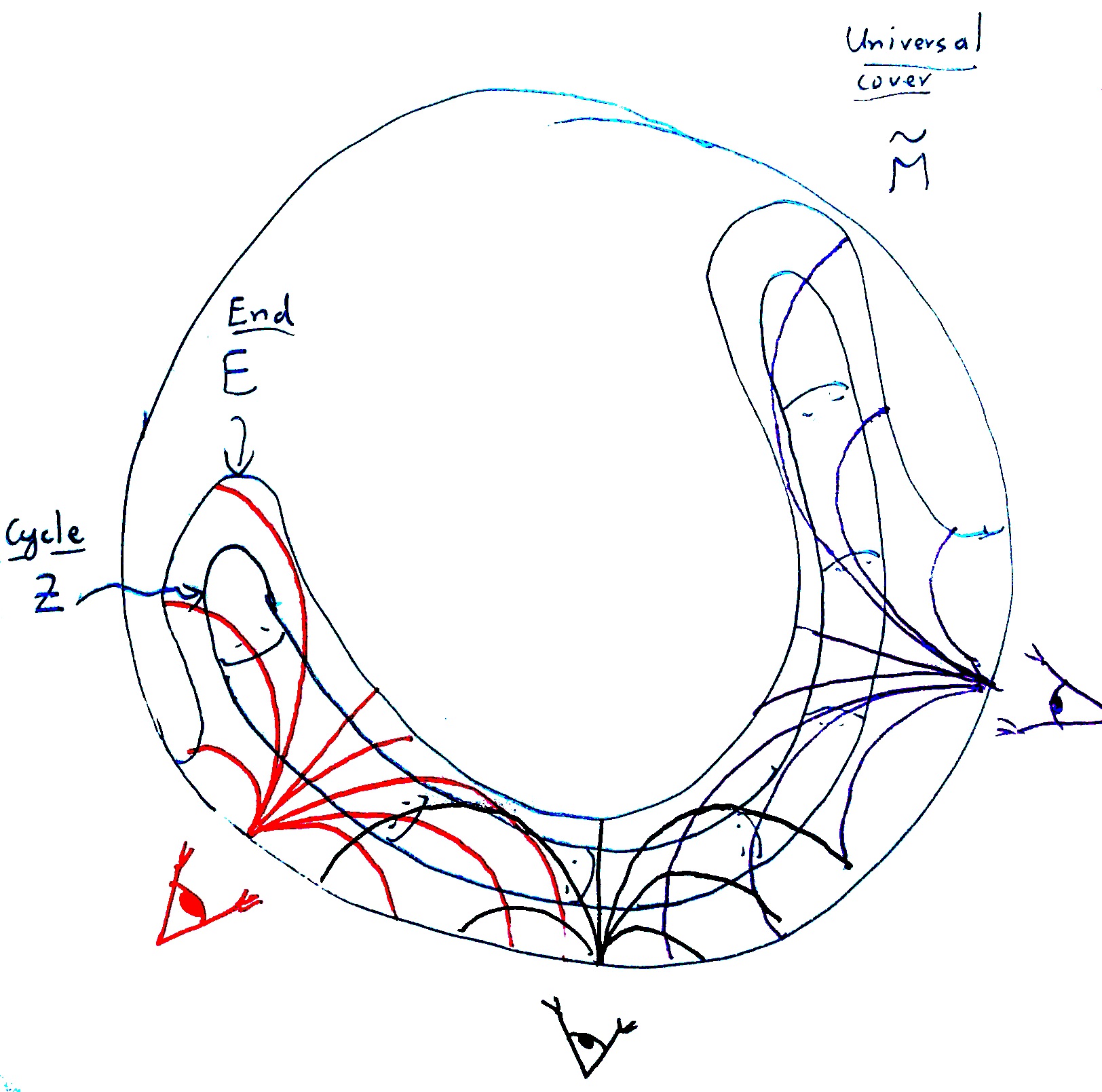}
\end{figure}

\subsection*{Contents} 
The first section sets up some notation and reviews background on nonpositively curved geometry. The second section contains background on algebraic topology along with the first instance of the fundamental ``local vanishing lemma'' (Proposition \ref{homologylemma}). The third section defines clumps. The fourth section proves the main result of the paper (Theorem \ref{maintheorem2}). The fifth section deals with coefficients and proves Theorem \ref{maintheorem3}. The last section contains a list of open questions. 

\subsection*{Acknowledgments}
I would like to thank T$\hat{\mathrm{a}}$m Nguy$\tilde{\hat{\mathrm{e}}}$n Phan for explaining her proof of Theorem \ref{tamtheorem}, discussing numerous aspects of the present paper and making crucial suggestions (e.g. the idea of how to organize patches into maximal clumps).

\section{Nonpositively curved geometry}
Good references for nonpositively curved geometry are the books \cite{ballman,ballmangromovschroeder,bridsonhaefliger} and \cite{eberlein}.
\subsection{Nilpotent groups and groups generated by small elements}

Let $N_1=N$ and $N_{i+1}=[N_i,N]$. The group $N$ is {\it nilpotent} if $N_m=0$ for some $m$. Associated to $N$ is the abelian group $\oplus_{i=1}^{m-1} N_i/N_{i+1}$. The {\it rank} of $N$ is the rank of the free abelian part of $\oplus_{i=1}^{m-1} N_i/N_{i+1}$.

\begin{lemma}[The Margulis lemma \cite{ballmangromovschroeder}]\label{lll}
Let $\widetilde{M}^n$ be a complete, simply connected manifold with sectional curvature $-1\leq K\leq 0$. There exist constants $\lambda_n$ and $I_n$ such that if $x\in \widetilde{M}$ and $\Gamma$ is a discrete group of isometries of $\widetilde{M}$ generated by isometries $\{\gamma_i\}$ with $d(x,\gamma_ix) < \lambda_n$, then $\Gamma$ has a finite index normal nilpotent subgroup of index $\leq I_n$.
\end{lemma}

\begin{lemma}[Small central elements]
\label{center}
Let $N$ be a nilpotent group with $N_m=0$. If elements $\{\gamma_i\}$ generate $N$ then there is a nontrivial element $\rho$ in the center of $N$ that can be expressed as a word in the generators $\{\gamma_i\}$ of length at most $3^m$. 
\end{lemma}
\begin{proof}
Either $\rho_1:=\gamma_1$ is in the center or there $\gamma_{i_2}$ such that $\rho_2:=[\gamma_1,\gamma_{i_2}]\not=1$. Either this is in the center or 
there is an element $\gamma_{i_3}$ such that $\rho_3:=[[\gamma_1,\gamma_{i_2}],\gamma_{i_3}]\not=1$ etc. The word length $||\rho_i||_{\{\gamma_i\}}$ of $\rho_i$ in the generating set $\{\gamma_i\}$ is $\leq 2(||\rho_{i-1}||_{\{\gamma_i\}}+1)$ so one easily checks that $||\rho_i||_{\{\gamma_i\}}<3^m$.  Note that $\rho_i\in N_i$ so this process terminates after at most $m$ steps, and the last non-identity $\rho_i$ is contained in the center. 
\end{proof} 
\begin{remark}
It follows from Malcev's theorem that for a torsionfree, finitely generated, nilpotent group $N$, if $N$ has rank $n$ then $N_n=0$.
\end{remark}

The following observation about isometric actions follows from the triangle inequality.

\begin{lemma}
\label{product}
Suppose that elements $\gamma_i$ act by isometries. Then
$$
d(\gamma_1\gamma_2\cdots\gamma_r x, x)\leq\sum_{i=1}^rd(\gamma_ix,x).
$$
\end{lemma}
\begin{proof}
\begin{eqnarray*}
d(\gamma_1\gamma_2\cdots\gamma_r x,x)&\leq&d(\gamma_1\gamma_2\cdots\gamma_r x,\gamma_1 x)+d(\gamma_1x,x),\\
&=&d(\gamma_2\cdots\gamma_r x,x)+d(\gamma_1x,x),\\
&\leq&\dots,\\
&\leq&\sum_{i=1}^rd(\gamma_ix,x).
\end{eqnarray*}
\end{proof}

\subsection{\label{horosphericalprojection}The geodesic projection onto a horosphere}
A general reference for this subsection is Eberlein's book \cite{eberlein}. 
Let $\widetilde M$ be a simply connected, nonpositively curved manifold. 
The horosphere $H(x,\xi)$ centered at a point at infinity $\xi\in\partial_{\infty}\widetilde M$ and passing through the point $x\in\widetilde M$ 
can be constructed in the following way. Let $r:[0,\infty)\ra\widetilde M$ be a geodesic ray starting at the point $r(0)=x$ and pointing towards $r(\infty)=\xi$, and let $B_t(r(t))$ be the ball of radius $t$ centered at the point $r(t)$. The union of these balls for all $t>0$ is the horoball $HB(x,\xi)=\cup_{t>0}B_t(r(t))$ and its boundary is the horosphere $H(x,\xi)$. Any such horosphere is homeomorphic to $\mathbb R^{n-1}$.  The entire space $\widetilde M$ decomposes topologically (but not isometrically!) as a product $\widetilde M=\mathbb R\times H(x,\xi)$, where the horizontal slices $\mathbb R\times y$ are geodesic lines passing through the point $y\in H(x,\xi)$ with one endpoint being the point at infinity $\xi$, and the vertical slices $t\times H(x,\xi)$ are horospheres $H(r(t),\xi)$. (The point $r(t)$ for negative $t$ is obtained by extending the geodesic ray $r$ to a geodesic line $r:(-\infty,\infty)\ra\widetilde M$.) The projection onto the second factor $p_{\xi}:\widetilde M\ra H(x,\xi)$ is called {\it the geodesic projection onto a horosphere centered at $\xi$.} The geodesic projection has the following important {\it shrinking property}: If $t<0$ then $d(0\times x_1,0\times x_2)\leq d(t\times x_1,t\times x_2)$.

\subsection{Sublevel sets of displacement functions}
For an isometry $\gamma$ of $\widetilde M$, look at the set $U:=\{x\in\widetilde M\mid d(\gamma x,x)<\varepsilon\}$ of points the isometry moves by less than $\varepsilon$. The set $U$ is convex and consequently contractible. Moreover, suppose that $\xi\in\partial_{\infty}\widetilde M$ is a point at infinity fixed by the the isometry $\gamma$, and let $r:[0,\infty)\ra\widetilde M$ be a geodesic ray pointing towards $r(\infty)=\xi$. Then for any $t>0$ we have $d(r(t),\gamma r(t))\leq d(r(0),\gamma r(0))$ since the rays $r$ and $\gamma r$ have the same endpoint $\xi$ and the geodesic projection to this point is shrinking. Consequently, if $x\in U$  then the entire geodesic ray $\overline{x\xi}$ connecting $x$ to the point at infinity $\xi$ is contained in $U$.

\subsection{Parabolic isometries and fixed points on the sphere at infinity}
An isometry $\phi$ of a simply connected nonpositively curved manifold $\widetilde M$ is {\it parabolic} if the displacement function $d_{\phi}:\widetilde M\ra\mathbb R,d_{\phi}(x)=d(x,\phi x)$ does not attain its infimum. A parabolic isometry $\phi$ fixes at least one point at infinity, and possibly more. Moreover, the following is a special case of the lemma at the end of section 3.9 of Ballmann-Gromov-Schroeder. 
\begin{lemma}
\label{fixedpoint}
Let $\phi$ be a parabolic isometry. There are points $\xi\in\partial_{\infty}\widetilde M$ at infinity such that every isometry $\gamma$ commuting with $\phi$ preserves $\xi$ and every horosphere $H(x,\xi)$ centered at $\xi$. 
\end{lemma}
\begin{proof}
Since $\phi$ is parabolic, the displacement function $d_{\phi}$ is a convex function that does not attain its infimum. It is $\gamma$-invariant since $d_{\phi}(\gamma x)=d(\gamma x,\phi\gamma x)=d(\gamma x,\gamma\phi x)=d(x,\phi x)=d_{\phi}(x)$.
Thus, in the notation of section 3.9, the set $\partial_2d_{\phi}$ is a nonempty subset of $Bd(X)$ (= the sphere at infinity $\partial_{\infty}\widetilde M$ by section 3.6). For every point $\xi$ in the set $\partial_2d_{\phi}$, the isometry $\gamma$ fixes $\xi$ and every horosphere $H(x,\xi)$ centered at $\xi$.  
\end{proof}

\begin{remark}
Eberlein defines the distinguished center of gravity of a parabolic isometry $\phi$ (page 278 of \cite{eberlein}). We do not know if every other isometry commuting with $\phi$ preserves all horospheres centered at this distinguished center of gravity, which is why we do not use it and rely on the above lemma instead. 
\end{remark}

\subsection{Small central parabolic elements}
\begin{lemma}[Sec. 7.5 of \cite{ballmangromovschroeder}]
If the commutator $[a,b]$ commutes with $a$ and $b$ then $\inf d_{[a,b]}=0$.
\end{lemma}
\begin{corollary}
\label{paraboliccenter}
Suppose $N=\left<\gamma_1,\dots,\gamma_k\right>$ is a nilpotent group of covering translations with $N_m=0$. If $N$ contains a parabolic element, then it contains a central parabolic element $\rho$ of word length $\leq 3^m$ in the generating set $\{\gamma_i^{\pm1}\}$.
\end{corollary}
\begin{proof}
This is clear if $N$ is abelian. If not, then the argument of Lemma \ref{center} gives a non-trivial central element $\rho=[a,b]$ of word length $\leq 3^m$. By the above lemma $\inf d_{\rho}=0$. Since $\rho$ acts by covering translations, it does not fix any points so it must be parabolic. 
\end{proof}

\subsection{Semisimple isometries in nilpotent groups}
An isometry $\phi$ is called {\it semisimple} if its displacement function $d_{\phi}$ does attain its infimum. The set of points in $\widetilde M$ where the infimum is attained is called the {\it minset} of $\phi$ and denoted $Min(\phi)$. If the group $N$ consists entirely of semisimple isometries, then the intersection $\cap_{\phi\in N}Min(\phi)$ is called the {\it minset of $N$}, and denoted $Min(N)$. 
\begin{lemma}[Sec. 7C of \cite{ballmangromovschroeder}]
\label{semisimple}
Suppose $N$ is a rank $r$ nilpotent group acting discretely and freely by isometries on $\widetilde M$. Then the subset of semisimple isometries $A<N$ is actually a central free abelian ($\cong\mathbb Z^r$) subgroup. Its minset is non-empty closed, convex subset which splits off a Euclidean deRham factor of rank $r'\geq r$, i.e. $Min(A)=C'\times\mathbb R^{r'}$. The $A$-action on this minset also splits. It is trivial on the first factor and is an action by translations on the second factor. 
\end{lemma}

\subsection{The closest point projection onto a closed, convex set}
Let $C$ be a closed, convex subset of $\widetilde M$. Then, the {\it closest point projection} $p_C:\widetilde M\ra C$ sends every point of $\widetilde M$ to the point of $C$ that is closest to it. Since $\widetilde M$ is simply connected and nonpositively curved and $C$ is convex, this map is well-defined. The closest point projection is a contracting map, i.e. $d(p_C(x),p_C(y))\leq d(x,y)$. Moreover, if $x$ is projected onto the point $p_C(x)$ then all points on the geodesic segment $\overline{xp_C(x)}$ get projected onto $p_C(x)$. Let us also note that if the convex set $C$ is invariant under an isometry $\gamma$ then $p_C(\gamma x)=\gamma p_C(x)$ since the closest point projection is defined in terms of the metric. Let $U:=\{x\in\widetilde M\mid d(\gamma x,x)<\varepsilon\}$ be the set of points the isometry $\gamma$ moves by less than $\varepsilon$.
\begin{lemma}
Suppose $C$ is a closed, convex subset of $\widetilde M$ invariant under an isometry $\gamma$ of $\widetilde M$. Then $\inf d_{\gamma}=\inf d_{\gamma|_C}$ and image under the closest point projection $p_C$ of the sublevel set $U$ and the minset of $\gamma$ is given by
\begin{eqnarray*}
p_C(U)&=&C\cap U,\\
p_C(Min(\gamma))&=&C\cap Min(\gamma),\\
&=&Min(\gamma|_C).
\end{eqnarray*}
\end{lemma}
\begin{proof}
The inequality 
$$
d(p_C(x),\gamma p_C(x))=d(p_C(x),p_C(\gamma x))\leq d(x,\gamma x)
$$ 
shows that $p_C(U)=C\cap U$ and $p_C(Min(\gamma))=C\cap Min(\gamma)$. It also shows that $\inf d_{\gamma}=\inf d_{\gamma\mid_C}$ which implies that $C\cap Min(\gamma)=Min(\gamma\mid_C)$.   
\end{proof}

\subsection{Splitting the minset and isometries acting on it}
We need the following description of intersections of minsets in the semisimple portion of the argument (see subsection \ref{semisimpleassembly}).
\begin{lemma}
\label{semisimplesplittinglemma}
Suppose $A$ and $B$ are commuting abelian groups of semisimple isometries and let $r$ be the rank of $A$. Then the minset of $A$ splits as $Min(A)=C\times\mathbb R^r$, $B$ acts on this minset and its action also splits. The closest point projection of the minset of $B$ to the minset of $A$ can be expressed as  
$$
p_{Min(A)}(Min(B))=Min(A)\cap Min(B)=D\times\mathbb R^r
$$ 
where $D$ is the minset of the $B$-action on $C$.
\end{lemma}
\begin{proof} 
Write $A=\mathbb Z^r\times F$ where $F$ is a finite abelian group. 
\begin{itemize}
\item
The finite group $F$ fixes $Min(A)$. 
\item
Since $A$ and $B$ commute
$B$ acts on $Min(A)$, so the closest point projection maps $p_{Min(A)}(Min(B))=Min(A)\cap Min(B)$ showing that the minsets of $A$ and $B$ intersect. 
\item
The minset $Min(A)$ splits as a product $C'\times\mathbb R^{r'}$
where $r'\geq r$ and the second factor is the Euclidean part of the deRham decomposition. Any isometry of the minset splits. 
\item
$B$ acts on $Min(A)$ (since $A$ and $B$ commute) so it also acts on $\mathbb R^{r'}$.
\item
The Euclidean part splits orthogonally as $\mathbb R^{r'}=\mathbb R^{r'-r}\times\mathbb R^r$ into convex hulls of $A$-orbits (each $\cong\mathbb R^r$) parametrized by $\mathbb R^{r'-r}$. 
\item
The $B$-action on $\mathbb R^{r'}$ sends convex hulls of $A$-orbits to convex hulls of $A$-orbits\footnote{Since $bAx=Abx$.}, so it defines an action $x\mapsto \gamma\cdot x$ on the space of such convex hulls $\mathbb R^{r'-r}$ sending $\{x\}\times\mathbb R^r$ to $\{\gamma\cdot x\}\times\mathbb R^r$.  
\item
Since $B$ acts by isometries on the Euclidean space $\mathbb R^{r'}$ it also sends any orthogonal complement $\mathbb R^{r'-r}\times y$ of $\{x\}\times\mathbb R^r$ to another orthogonal complement $\mathbb R^{r'-r}\times\{y'\}$ and in this way defines a $B$-action $y\mapsto \gamma\cdot y$ on the space of orthogonal complements $\mathbb R^{r'-r}$. 
\item
So, the $B$-action on $\mathbb R^{r'}=\mathbb R^{r'-r}\times\mathbb R^r$ splits as $\gamma(x,y)=(\gamma\cdot x,\gamma\cdot y)$. 
\item
Since $Min(A)\cap Min(B)$ is a convex set invariant under the group $\mathbb Z^r$, it splits as $D\times\mathbb R^r$, where $D$ is a convex subset of $C:=C'\times\mathbb R^{r'-r}$. 
\item
By the previous lemma, the minset $Min(B|_{Min(A)})$ of $B$ acting on $Min(A)$ is the intersection $Min(A)\cap Min(B)=D\times\mathbb R^r$. Since the $B$-action on $Min(A)=C\times\mathbb R^r$ splits, this minset also splits as the product of the minset of $B$ acting on $C$ with the minset of $B$ acting on $\mathbb R^r$, which turns out to be all of $\mathbb R^r$ (because it is an $A$-invariant convex set).
\end{itemize}  
\end{proof}

\section{Algebraic topology}
Good references for this section are Section 4.G of \cite{hatcher} and Appendix E of \cite{davisbook}.
\subsection{The nerve of a cover}
Let $\{X_i\}$ be an open cover of a topological space $X$, i.e. $X = \displaystyle{\cup_i X_i}$. The \emph{nerve} of this cover is a simplicial complex $N$ defined as follows.
\begin{itemize}
\item The vertices of $N$ correspond to the sets $X_i$, so they are labeled by the indices $i$. 
\item If for some $i_0, ..., i_j$, the intersection $X_{i_0}\cap ... \cap X_{i_j}\ne \emptyset$, then there is a $j$-simplex $\alpha$ in $N$ with vertices $i_0,\dots,i_j$. We denote the corresponding intersection by $X_{\alpha}=\cap_{i\in\alpha}X_{\alpha}$. 
\end{itemize} 
\subsection{The barycentric subdivision of the nerve}
The barycentric subdivision $\mathcal X$ of the nerve of $\{X_i\}$ has a useful description in terms of increasing chains of simplices. Namely, a $k$-simplex in the barycentric subdivision is given by a strictly increasing chain 
\begin{equation}
\label{simplexchain}
\alpha_0\subset\alpha_1\subset\dots\subset\alpha_k
\end{equation}
of simplices in the nerve, and the faces of the $k$-simplex are subchains.  
Corresponding to this $k$-simplex is a decreasing (but maybe not strictly decreasing!) chain of subsets 
\begin{equation}
\label{subsetchain}
X_{\alpha_0}\supset X_{\alpha_1}\supset\dots\supset X_{\alpha_k}.
\end{equation}
\subsection{The reduced nerve}
If the open cover $\{X_i\}$ is locally finite then the barycentric subdivision $\mathcal X$ deformation retracts onto a subcomplex $\overline{\mathcal X}$ 
which has the property that for any $k$-simplex in $\overline{\mathcal X}$ the corresponding chain of subsets (\ref{subsetchain}) is strictly decreasing. 
The subset $\overline{\mathcal X}$ is called the {\it reduced nerve} of the cover $\{X_i\}$. It can be described in the following way. 

For any simplex $\alpha$ in the nerve, let 
$$
\overline{\alpha}=\{i\mid X_i\supset X_{\alpha}\}.
$$ 
This is a (possibly bigger) simplex in the nerve\footnote{The cover needs to be locally finite so that each $\overline{\alpha}$ is finite.}, since $\cap_{i\in\overline{\alpha}}X_i=X_{\alpha}\not=\emptyset$. Now, look at the map
\begin{eqnarray*}
p:\mathcal X&\ra&\mathcal X,\\
\alpha_0\subset\alpha_1\subset\dots\subset\alpha_k&\mapsto&\overline{\alpha_0}\subset\overline{\alpha_1}\subset\dots\subset\overline{\alpha_k},
\end{eqnarray*}
where we throw out, if necessary, any repetitions in the chain $\overline{\alpha_1}\subset\dots\subset\overline{\alpha_k}$. The map $p$ is retraction onto its image, its fibres are contractible, and it is not hard to see that it can be extended to a deformation retraction\footnote{We don't actually need to use this last fact.}. The reduced nerve is the image of this map 
$$
\overline{\mathcal X}:=p(\mathcal X).
$$

\begin{figure}
\centering
\includegraphics[scale=0.23]{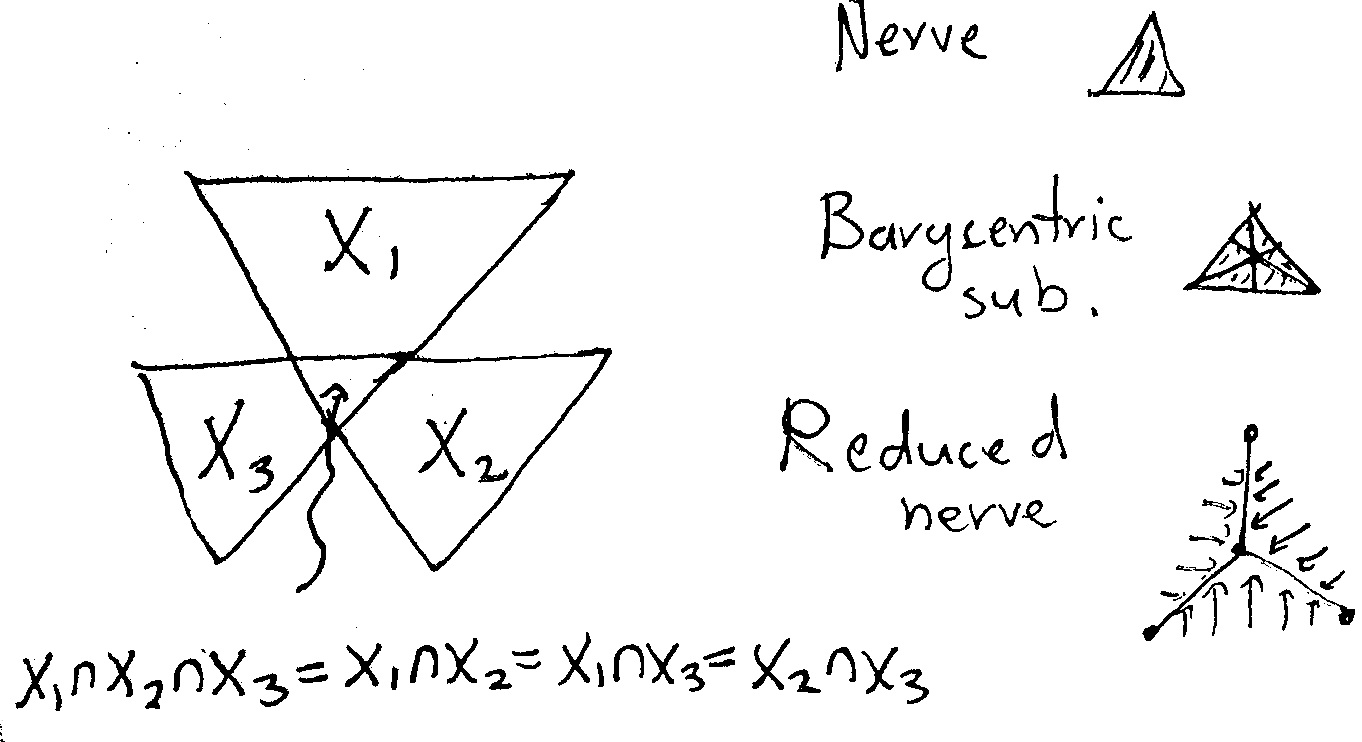}
\end{figure}

By construction, if $\alpha$ is a subsimplex of $\beta$ such that $\overline{\alpha}$ is a proper subsimplex of $\overline{\beta}$ then there is $X_i$ that contains $X_{\overline{\beta}}$ but does not contain $X_{\overline{\alpha}}$, which implies $X_{\overline{\alpha}}$ is strictly bigger than $X_{\overline{\beta}}$.

\subsection{\label{Delta} Relating the nerve of the cover $\{X_i\}$ to the space $X$ via a fattening $\Delta X$}
The nerve is related to the space $X$ via an auxiliary space $\Delta X$ that is a sort of ``fattened'' version of $X$. The space $\Delta X=\coprod_{\alpha}(X_{\alpha}\times\alpha)/\sim$ is the quotient space of the disjoint union of all the possible products $X_{\alpha}\times\alpha$ by identifications along the faces $\partial_j\alpha$ of the simplices $\alpha$ via the inclusions $X_{\alpha}\times \partial_j\alpha\hookrightarrow X_{\partial_j\alpha}\times\partial_j\alpha.$ 
\begin{itemize}
\item
Note that we have a collapse map $\Delta X\ra X$ obtained by collapsing all the simplices $\alpha$. This map is a homotopy equivalence. (See 4G.2 of \cite{hatcher}.)
\item
There is a map $\Delta X\ra\mathcal X$ from the fattening to the barycentric subdivision of the nerve, induced by sending $X_{\alpha}\times\alpha$ to $\alpha$.
\item
There is also a reduced thickening $\overline\Delta X$ and a map $\Delta X\ra\overline\Delta X$ obtained by first taking the barycentric subdivisions 
$$
\coprod_{\alpha_0\subset\dots\subset\alpha_i} X_{\alpha_i}\times\{\alpha_0\subset\dots\subset\alpha_i\}/\sim
$$
of the simplices $\alpha$ occurring in $\Delta X$ 
(geometrically, this process doesn't change anything) and then applying the map $p$ to get 
$$
\overline{\Delta} X:=\coprod_{\overline{\alpha_0}\subset\dots\subset\overline{\alpha_i}} X_{\alpha_i}\times\{\overline{\alpha_0}\subset\dots\subset\overline{\alpha_i}\}/\sim.
$$ 
The reduced thickening $\overline{\Delta}X$ can also be described as the inverse image of the reduced nerve $\overline{\mathcal X}$ under the projection $\Delta X\ra\mathcal X$. 
\end{itemize}
\begin{lemma}\cite[Proposition 4G.1]{hatcher}
\label{delta}
Suppose $\{X_i\}$ is an open cover of $X$. If every finite intersection $X_{\alpha}$ in this cover is either empty or contractible then the projection $\Delta X\ra\mathcal X$ is a homotopy equivalence.  
\end{lemma}
Since the projection $\Delta X\ra X$ is always a homotopy equivalence (Proposition 4G.2 of \cite{hatcher}) and the nerve of the cover $\{X_i\}$ is homeomorphic to its barycentric subdivision $\mathcal X$, we get the following corollary. 

\begin{corollary}\cite[Cor. 4G.3]{hatcher}
\label{hequiv}
Let $\{X_i\}$ be an open cover of $X$ whose finite intersections are either empty or contractible. Then $X$ is homotopy equivalent to the nerve of the cover $\{X_i\}$.
\end{corollary}


\subsection{Mayer-Vietoris assembly}
For future use we mention that even if the finite intersections of a cover do have some non-trivial topology, the homology of the reduced nerve can sometimes be related to the homology of the cover via a Mayer-Vietoris assembly argument which can be packaged into the spectral sequence
$$
\label{mvss}
E^2_{j,k}=H_{k}(\overline{\mathcal X};H_j(X_{\alpha}))\implies H_{j+k}(\cup_iX_i).
$$
This works when $\{X_i\}$ is a cover of a space by open sets and also when $\{X_i\}$ is a cover of a CW-complex by subcomplexes.
The following lemma illustrates how this can be used. 
\begin{lemma}
\label{basicassembly}
If for every $k$-simplex $\alpha$ in the reduced nerve we have
$\overline H_{\geq n-k}(X_{\alpha})=0$, then $H_{\geq n}(\overline{\mathcal X})=0$ implies $H_{\geq n}(\cup_iX_i)=0$.  
\end{lemma}
\begin{proof}
The condition $H_{\geq n-k}(X_{\alpha})=0$ implies that for $j+k\geq n$ the only possible non-zero $E^2_{j,k}$ terms occur on the $j=0$ line and are $E^2_{0,k}=H_k(\overline{\mathcal X})$. So, the spectral sequence implies that in dimensions $\geq n$ the homology of the union $\cup_iX_i$ is bounded above by the homology of the reduced nerve $\overline{\mathcal X}$. The lemma follows from this.  
\end{proof}


\subsection{\label{htriv}Homotopically trivial group actions}
Next, suppose that a group $N$ acts on $X$ by covering translations and preserves each of the individual sets $X_i$. 
Then, the $N$-action fits into the commutative diagram
$$
\begin{array}{ccc}
X&\stackrel{\cdot n}\ra &X\\
\uparrow&&\uparrow\\
\Delta X&\stackrel{\cdot n}\ra&\Delta X\\
\downarrow&&\downarrow\\
\mathcal X&\stackrel{\cdot n}=&\mathcal X,
\end{array}
$$
with the vertical maps being homotopy equivalences, and the horizontal maps the actions of an element $n\in N$. The bottom map is just the identity so if $\Delta X\ra\mathcal X$ is a homotopy equivalence, then the top map is homotopically trivial. Consequently, we get the following corollary. 
\begin{corollary}
Suppose we have an open cover $X=\cup X_i$ in which every finite intersection is either empty or contractible. If a group $N$ acts on $X$ by covering translations and preserves each of the individual $X_i,$ then its action on $X$ is homotopically trivial. In particular, $N$ acts trivially on the homology groups $H_*(X)$. 
\end{corollary}
This can be used to show that some homology cycles in $X$ bound. The local vanishing lemma below is one such result that is a crucial ingredient in the proof of the main theorem.


\subsection{Local vanishing lemma}
\begin{proposition}
\label{homologylemma}
Let $W$ be a connected open submanifold of $\mathbb R^{n-1}$. 
Suppose that a nilpotent group $N$ of rank $r$ acts on $\mathbb R^{n-1}$ by covering translations, preserves $W$ and acts homotopically trivially on $W$. Then 
\begin{equation}
\overline H_{\geq n-1-r}(W)=0.
\end{equation} 
Moreover, if $r=n-1$ then $W$ is all of $\mathbb R^{n-1}$. 
\end{proposition}
\begin{proof}
The spectral sequence corresponding to the bundle 
\begin{equation}
W\ra (W\times\mathbb R^{n-1})/N\ra\mathbb R^{n-1}/N
\end{equation}
is
\begin{equation}
E^2_{i,j}=H_i(N;H_j(W))\implies H_{i+j}((W\times\mathbb R^{n-1})/N)=H_{i+j}(W/N),
\end{equation}
The $N$-action on $H_j(W)$ is trivial because the $N$-action on $W$ is homotopically trivial. Let $k$ be the largest integer for which $H_k(W)\neq 0$. Since $N$ acts by covering translations on $\mathbb R^{n-1}$, it is torsionfree. Since $N$ is a nilpotent group of rank $r$, it is the fundamental group of a closed aspherical $r$-manifold, and passing to an index two subgroup of $N$ if necessary, we may assume this manifold is orientable. Consequently, there is a nontrivial fundamental class in $H_r(N)$, and hence also in $H_r(N;H_k(W))$. Thus, the $E^2$-term $E^2_{r,k}=H_r(N;H_k(W))$ is nonzero. Moreover, this term never gets killed in the spectral sequence and hence $H_{r+k}(W/N)\neq 0$. Since $W$ is an open submanifold of $\mathbb R^{n-1}$ we must have one of the following two possibilities. 
\begin{itemize}
\item
$r+k=n-1$. In this case $W/N$ is a closed $(n-1)$-manifold and consequently $W$ is also a closed $(n-1)$-manifold. So in fact $W=\mathbb R^{n-1}$. Consequently $k=0$ and $r=n-1$. 
\item $r+k<n-1$, which implies that $H_{\geq n-1-r}(W)=0$. 
\end{itemize}
In both cases, we have $\overline H_{\geq n-1-r}(W)=0$, which is what we needed to prove.
\end{proof}

\begin{remark} Here is the simplest instance of this proposition. Suppose $W$ is an open subset of the plane $\mathbb R^2$ and $\rho_t:\mathbb R^2\times[0,1]\ra\mathbb R^2$ is a homotopy preserving $W$(so $\rho_t(W)\subset W$) from the identity $\rho_0=id$ to a covering translation $\rho_1$. If a loop $\alpha\subset W$ is not contractible in $W$ then there is a point $x$ inside $\alpha$ that is not in $W$. Since $\rho_1$ is a covering translation, for sufficiently large $K$ the translated loop $\rho_1^K\alpha$ doesn't contain $x$, so the homotopy $\rho^K_t(\alpha)$ moves the loop $\alpha$ off the point $x$ without ever passing through the point $x$. This cannot happen, so $\alpha$ must be contractible in $W$. 
\end{remark}


\section{\label{clumps}Clumps}
Suppose that we have a finite collection of subsets $\{X_i\}$ of $X$ and corresponding subgroups $\{G_i\}$ of some group $G$. For each simplex $\sigma$ in the nerve of the cover $\{X_i\}$ let $X_{\sigma}=\cap_{i\in\sigma}X_i$ and $G_{\sigma}=\left<G_i\mid i\in\sigma\right>$. Any union of some $X_{\sigma}$ is called a {\it clump}. We will only care about those clumps that can be defined by some group $N$ via $Y_N=\cup_{N<G_{\sigma}}X_{\sigma}$. The intersection of two such clumps is given by the {\it intersection formula} 
$$
Y_N\cap Y_M=Y_{\left<N,M\right>}.
$$
\begin{proof}
If a point is in $\cup_{\left<N,M\right><G_{\sigma}}X_{\sigma}$, then it is clearly in both $Y_N$ and $Y_M$. 
In the other direction, a point $x$ in the intersection $Y_N\cap Y_M$ is contained in some intersection $X_{\sigma}\cap X_{\tau}$ for $G_{\sigma}>N,G_{\tau}>M$. The intersection $X_{\sigma}\cap X_{\tau}$ is nonempty, so it corresponds to a bigger simplex $\rho$ in the nerve containing both $\sigma$ and $\tau$. Thus, the point $x$ is in $X_{\rho}$ with $\left<N,M\right><\left<G_{\sigma},G_{\tau}\right><G_{\rho}$. This shows $x$ is contained in $\cup_{\left<N,M\right><G_{\sigma}}X_{\sigma}$, which is what we needed to prove.
\end{proof}
There may be more than one group that defines the same clump $Y$ but, by the intersection formula, there is always a largest such group, namely $\left<N\mid Y_N=Y\right>$. 

The group $N$ is called {\it minimal} (with respect to the collection $\{G_{\sigma}\}$) if for every $\sigma$ either
\begin{itemize}
\item
$N<G_{\sigma}$, or 
\item
$N\cap G_{\sigma}$ is an infinite index subgroup of $N$. 
\end{itemize}
\begin{lemma}
If $N$ and $M$ are minimal, then $\left<N,M\right>$ is minimal.
\end{lemma}
\begin{proof}
If $G_{\sigma}\cap\left<N,M\right>$ is a finite index subgroup of $\left<N,M\right>$ then $G_{\sigma}\cap N$ is a finite index subgroup of $N$ so, since $N$ is minimal, $N<G_{\sigma}$. Similarly $M<G_{\sigma}$. Consequently $\left<N,M\right><G_{\sigma}$, proving that $\left<N,M\right>$ is minimal. 
\end{proof}
A clump $Y$ is called {\it maximal} if it can be defined by an infinite minimal group $N$. There may be more than one minimal group defining the maximal clump $Y$ but, by the above lemma, there is always a largest such minimal group, namely $\left<N \mbox{ minimal }\mid Y_N=Y\right>$. This group is called {\it the minimal group corresponding to the maximal clump $Y$}. 

Given the data $\{(X_i,G_i)\}$, let $\{Y_{\alpha}\}$ be the collection of maximal clumps whose corresponding minimal groups $\{N_{\alpha}\}$ have the additional property\footnote{This connects properties of the $\{G_i\}$ to those of the $\{N_{\alpha}\}$ e.g. presence of parabolic elements or lower bounds on rank.}
\begin{itemize}
\item 
$N_{\alpha}$ virtually contains $G_i$ for some infinite group $G_i$. 
\end{itemize}
Then 
\begin{itemize}
\item
$\{Y_{\alpha}\}$ is closed under intersections,
\item
Every set $X_i$ with infinite group $G_i$ is contained in some maximal clump $Y_{\alpha}$\footnote{Defined by the minimal finite index subgroup of $G_i$, i.e. the intersection of all $G_{\sigma}$ that virtually contain $G_i$. Note that this group virtually contains $G_i$.}, so
$$
\bigcup_iX_i=\bigcup_{\alpha}Y_{\alpha}\cup \bigcup_{G_i \mbox{ finite}}X_i.
$$
\item 
if $Y_{\alpha}\supset Y_{\beta}$ then $N_{\alpha}<N_{\beta}$, and
\item
if $Y_{\alpha}\supsetneq Y_{\beta}$ then $N_{\alpha}$ is an infinite index subgroup of $N_{\beta}$.  
\end{itemize}
In our application the groups $\{N_{\alpha}\}$ will be nilpotent, so the last bullet leads to the following {\it growing ranks property} of maximal clumps.  
\begin{corollary}[Growing ranks]
\label{growingrank}
If all the $\{N_{\alpha}\}$ are nilpotent and $Y_{\alpha_0}\supset\dots\supset Y_{\alpha_k}$
is a strictly decreasing sequence of maximal clumps then $\mbox{rank}(N_{\alpha_k})\geq k+\mbox{rank}(N_{\alpha_0})$.
\end{corollary}  
\begin{proof}
An infinite index nilpotent subgroup has smaller rank, so the chain of maximal clumps in the statement of the corollary gives inequalities
$$
\mbox{rank}(N_{\alpha_0})<\dots<\mbox{rank}(N_{\alpha_k})
$$  
which imply $\mbox{rank}(N_{\alpha_k})\geq k+\mbox{rank}(N_{\alpha_0})$.
\end{proof}

\section{Proof of Theorem \ref{maintheorem2}}
\label{mainsection}
\subsection{\label{patches}$\varepsilon$-patches}
If $S_i$ is a finite set of isometries such that the $\varepsilon$-sublevel set 
$$
L_i:=\bigcap_{\gamma\in S_i}\left\{x\in\widetilde M\mid d(x,\gamma x)<\varepsilon\right\}
$$ 
is not empty, then we call $(S_i,L_i)$ an {\it $\varepsilon$-patch}. 
If $\varepsilon$ is less than the Margulis constant, then the group $\left<S_i\right>$ generated by the isometries is almost nilpotent and contains a nilpotent subgroup of index $\leq I_n$, so the group $\Gamma_i=\left<\gamma^{I_n!}\mid\gamma\in S_i\right>$ is actually nilpotent. In this case, the {\it rank} of the patch is the rank of the nilpotent group $\Gamma_i$. The patch is {\it parabolic} if $\Gamma_i$ contains parabolic elements. Otherwise, (if the entire group $\Gamma_i$ is semisimple) the patch is {\it semisimple}.
The $\varepsilon$-patch $(S_i,L_i)$ may not be invariant under the group $\Gamma_i$, but it is contained in a larger, $\Gamma_i$-invariant $\varepsilon'$-patch $(Z_i,L^z_i)$, where $\varepsilon'=I_n!3^n\varepsilon$. This larger patch is defined in the following way. Let $S'_i=\{\gamma^{I_n!}\mid\gamma\in S_i\}$ be the generating set for $\Gamma_i$ and
$$
Z_i:=\left\{1\not=\gamma\in Z(\Gamma_i)\mid ||\gamma||_{S'_i}\leq 3^n\right\}
$$
the set of non-trivial isometries commuting with $\Gamma_i$ and of length $\leq 3^n$ in this generating set $S'_i$. The set $Z_i$ is non-empty by Lemma \ref{center}. The corresponding $\varepsilon'$-sublevel set 
$$
L^z_i:=\bigcap_{\gamma\in Z_i}\left\{y\in\widetilde M\mid d(y,\gamma y)<\varepsilon'\right\}
$$
contains $L_i$ by Lemma \ref{product}. 
Finally, the resulting $\varepsilon'$-patch $(Z_i,L^z_i)$ is $\Gamma_i$-invariant because $Z_i$ commutes with $\Gamma_i$.

\subsection{\label{end}The end}
Recall that $M$ is the interior of a compact manifold with boundary $\partial M$. So, $M$ has a neighborhood of infinity of the form $\partial M\times[0,\infty)$. 
This neighborhood lifts to $\partial\widetilde M\times[0,\infty)$ in the universal cover $\widetilde M$. 
Let $E$ be a single component of this, and $E_t$ the corresponding component of $\partial\widetilde M\times[t,\infty)$. The fundamental group $\Gamma:=\pi_1M$ permutes the components of $\partial\widetilde M\times[0,\infty)$, so if an element of the fundamental group $\gamma\in\Gamma$ does not move $E$ completely off itself, ($\gamma E\cap E\not=\emptyset$) then in preserves $E$ and its boundary $\partial E$ (so that $\gamma E=E$ and $\gamma\partial E=\partial E$). For convenience, we reparametrize so that $d(\partial E_0,E_t)\geq t$, so that the $R$-neighborhood $N_R(\partial E_0)$ of $\partial E_0$ does not meet $E_R$.

\subsection{The thin part}
The {\it $(\pi_1M,\varepsilon)$-thin} part of $\widetilde M$ is the subset 
$$
\widetilde M_{<\varepsilon}:=\{x\in\widetilde M\mid d(x,\gamma x)<\varepsilon \mbox{ for some } \gamma\in\pi_1M\setminus\{1\}\}
$$ 
of points that are moved less than $\varepsilon$ by some nonidentity element of the fundamental group $\Gamma:=\pi_1M$. We will usually just call this the $\varepsilon$-thin part of $\widetilde M$. The $\varepsilon$-thin part projects to the subset of $M$ on which the injectivity radius is $<\varepsilon/2$. 
Pick a constant $\mu$ for which the subset of injectivity radius $<\mu/2$ is contained in the neighborhood of the end $\partial M\times[0,\infty)$. On the universal cover $\widetilde M$, this means the $\mu$-thin part $\widetilde M_{<\mu}$ is a subset of $\partial\widetilde M\times[0,\infty)$. Further, we pick $\mu<\lambda_n$ to be less than the Margulis constant $\lambda_n$. 

\subsection{\label{rankn-1}The rank $n-1$ lemma} If one of the groups $\Gamma_i$
corresponding to an $\varepsilon$-patch (with small enough $\varepsilon$) has rank $n-1$ then the situation is much simpler that the general case. Projecting to an appropriate horosphere or minset, one gets a cocompact $\Gamma_i$-action and uses this to show the entire component of the end $E$ containing $L_i$ is contractible. There is no need to assemble any patches together.

More precisely, let $\varepsilon'=I_n!3^n\varepsilon$.
If $\varepsilon'<\mu$ then both the $\varepsilon$-patch $L_i$ and also the larger $\Gamma_i$-invariant $\varepsilon'$-patch $L^z_i$ are contained in some component $E$ of $\partial \widetilde M\times[0,\infty)$.  
\begin{lemma}[Rank $n-1$ lemma]
\label{rankn-1lemma}
If the $\varepsilon$-patch $L_i$ has rank $n-1$ then $E$ is contractible. 
\end{lemma}
\begin{proof}
There are two cases to consider.

\noindent
{\bf Case 1: $L_i$ is parabolic.} We pick a nontrivial parabolic element $\gamma$ in the center of $\Gamma_i$ and let $\xi$ be a point at infinity as in Lemma \ref{fixedpoint}. Look at the geodesic projection $p_{\xi}:L^z_i\ra H(x,\xi)\cong\mathbb R^{n-1}$ to a horosphere at this point at infinity. The fibres of this projection are connected geodesics (either open rays or entire lines) pointing at $\xi$ so the image $p_\xi(L^z_i)\subset\mathbb R^{n-1}$ is a contractible and open set.
The group $\Gamma_i$ preserves the set $L^z_i$ and also the point at infinity $\xi$, so it preserves the image $p_{\xi}(L^z_i)$. In summary $p_{\xi}(L^z_i)\subset\mathbb R^{n-1}$ is an open contractible subset invariant under the rank $n-1$ group $\Gamma_i$, so it must be the entire $\mathbb R^{n-1}$. The fact that $\Gamma_i$ has rank $n-1$ also implies that its action on $H(x,\xi)\cong\mathbb R^{n-1}$ has compact fundamental domain $F$, which means that for sufficiently large $t$ the horosphere $H(r(t),\xi)$ is contained in $L^z_i$. 

\noindent
{\bf Case 2: $L_i$ is semisimple.} In this case all elements are semisimple, $\Gamma_i$ is a free abelian group of rank $n-1$ ($\Gamma\cong\mathbb Z^{n-1}$) and its minset splits as $C\times\mathbb R^{n-1}$ with $\Gamma_i$ acting trivially on $C$ and by covering translations on $\mathbb R^{n-1}$ (Lemma \ref{semisimple}). The image of the closest point projection to this minset $p_{Min(\Gamma_i)}(L^z_i)=L^z_i\cap Min(\Gamma_i)$ is a (non-empty) convex, $\mathbb Z^{n-1}$-invariant subspace, so it contains a flat $\mathbb R^{n-1}\subset L^z_i$.

In either case, we have an $\mathbb R^{n-1}\subset L^z_i\subset E_0$ (in the first case a horosphere, and in the second case a totally geodesic flat) that separates the universal cover into two contractible components $\widetilde M\setminus\mathbb R^{n-1}=H_1\coprod H_2$,  with the first component $H_1$ containing $\partial E_0$ and the second component $H_2$ contained entirely in $E_0$. 
We claim that    
\begin{itemize}
\item
for sufficiently large $R$, the translate of the end $E_R$ is contained in $H_2$.
\end{itemize}
Since the composition of the inclusions $E_R\subset H_2\subset E_0=E$ is a homotopy equivalence and $H_2$ is contractible, this will prove that $E$ is contractible.

Now, we prove the claim.
Recall that the $\Gamma_i$ action on $\mathbb R^{n-1}$ has a compact fundamental domain $F$ (because $\Gamma_i$ has rank $n-1$), so there is a constant $R>0$ such that the $R$-neighborhood $N_R(\partial E_0)$ of $\partial E_0$ in $E_0$ contains $F$. Since both $\partial E_0$ and the $\mathbb R^{n-1}$ are invariant under the group $\Gamma_i$ ($\partial E_0$ is invariant under $\Gamma_i$ because $\Gamma_i$ preserves $L^z_i$ and hence doesn't move $E_0$ completely off itself) we find that the entire $\mathbb R^{n-1}$ is contained in the $R$-neighborhood of $\partial E_0$. On the other hand, $E_R$ does not meet $N_R(\partial E_0)$, so it must be contained in $H_2$. 
\end{proof}

\subsection{Setup for the rest of the argument\label{smallepsilon}}
If there is an $\varepsilon$-patch of rank $n-1$ with $\varepsilon'<\mu$, then we have shown the component $E$ containing it is contractible, so we are done with the proofs of Theorems \ref{maintheorem2} and \ref{maintheorem3} for that component. So, for the rest of sections \ref{mainsection} and \ref{coversandcoefficients} we can and will assume that
\begin{itemize}
\item
all $\varepsilon$-patches with $\varepsilon'<\mu$ have rank $\leq n-2$.
\end{itemize}
Since $M$ has finite volume, its injectivity radius tends to zero. So, for any $\varepsilon>0$ we can find a translate of the end $\partial\widetilde M\times[t,\infty)$ that is entirely contained in the $\varepsilon$-thin part $\widetilde M_{<\varepsilon}$. Thus, 
\begin{itemize}
\item
for $c\in H_*(\partial\widetilde M)$ and $\varepsilon>0$, we can assume $c$ is covered by $\varepsilon$-patches of rank $\geq 1$. 
\end{itemize}
In the rest of the proof, we will take cycles in finite unions of $\varepsilon$-patches and fill them inside unions of larger patches. In order to be able to do this, the initial $\varepsilon$ needs to be taken sufficiently small. In fact, it turns out that
\begin{equation}
\varepsilon<{{\mu}\over (I_n!)^{2^{n-3}+2}3^n}
\end{equation}
will work. We begin by organizing unions of $\varepsilon$-patches into clumps. 

\subsection{\label{small}Small clumps of $\varepsilon$-patches}
The $\varepsilon$-patches are closed under intersections, in the sense that for any simplex $\sigma$ in the nerve of $\{L_i\}$ we get a $\varepsilon$-patch $(S_{\sigma},L_{\sigma})$ where $S_{\sigma}=\cup_{i\in\sigma}S_i$ and $L_{\sigma}=\cap_{i\in\sigma}L_i$. The corresponding nilpotent\footnote{if $\varepsilon<$ Margulis constant} group is $\Gamma_{\sigma}=\left<S'_{\sigma}\right>=\left<\Gamma_{i}\mid i\in\sigma\right>$. 
So, we can associate maximal clumps $\{Y_{\alpha}\}$ and minimal nilpotent groups $\{N_{\alpha}\}$ to the collection of sets $\{L_i\}$ and groups $\{\Gamma_i\}$ as described in section \ref{clumps}.

\subsection{\label{big}Big clumps}
The big clump corresponding to $N_{\alpha}$ is 
$$
Y^z_\alpha:=\bigcup_{N_{\alpha}<\Gamma_{\sigma}}L^z_{\sigma}.
$$
Note that if $Y_{\alpha}\subset Y_{\beta}$ then $N_{\alpha}>N_{\beta}$ so that $Y^z_{\alpha}\subset Y^z_{\beta}$. Moreover $Y^z_{\alpha}$ is $N_{\alpha}$-invariant, because all of the large patches $L^z_{\sigma}$ that make it up are $N_{\alpha}$-invariant. 
Also note that if the intersection of two small clumps is $Y_{\alpha}\cap Y_\beta=Y_{\gamma}$, then the intersection of big clumps $Y^z_{\alpha}\cap Y^z_{\beta}$ contains $Y^z_{\gamma}$,\footnote{Because $Y_{\alpha}\cap Y_{\beta}=Y_{\gamma}$ implies $\left<N_{\alpha},N_{\beta}\right> < N_{\gamma}$ and hence $Y^z_\alpha\cap Y^z_\beta\supset\cup_{\left<N_\alpha,N_\beta\right><\Gamma_{\sigma}}L^z_{\sigma}\supset Y^z_\gamma$.} but may, in general, be bigger than $Y^z_\gamma$. 

\begin{remark}
The purpose of the big clumps $Y^z_\alpha$ is to produce a cover by sets that are as large as possible, and each have a homotopically trivial action of a nilpotent group on them. Each one identifies a part that can be ``seen'' from a single point at infinity. This is made more precise in the local vanishing lemma below.
\end{remark}

\subsection{\label{parabolicclumps}Topology of big parabolic clumps}
\begin{lemma}[Local vanishing lemma]
\label{lvl}
Suppose $N$ is a nilpotent group of rank $r\leq n-2$ containing parabolic elements. If $N<\Gamma_i$ for all $i$ then 
$$
H_{\geq n-1-r}(\cup L_i^z)=0.
$$   
\end{lemma}

\begin{remark}
Note that Proposition \ref{homologylemma} implies $H_{\geq n-r}(\cup L_i^z)=0$. The point of the local vanishing lemma is that one can reduce the dimension in which the homology of a parabolic clump vanishes by one more by projecting onto an appropriate horosphere. 
\end{remark}

\begin{proof}
Since the group $N$ contains a nontrivial parabolic element, its center $Z(N)$ also contains a nontrivial parabolic element $\gamma$ (Corollary \ref{paraboliccenter}). Take a point $\xi$ on the sphere at infinity as in Lemma \ref{fixedpoint}. Then group $N$ and all the sets $Z_i$ commute with $\gamma$, so they fix $\xi$ and preserve all horospheres $H(x,\xi)$ centered at $\xi$.  
Let 
$$
p:\widetilde M\ra H(x,\xi)\cong\mathbb R^{n-1}
$$ 
be the geodesic projection onto a horosphere based at $\xi$. 
The fibres of $p$ on all the $L^z_i$ are connected geodesics (either open rays or entire lines) pointing at $\xi$. Note that
\begin{itemize}
\item The projections $\{p(L^z_i)\}$ form an open cover of $p(\cup L^z_i)$.
\item All intersections from the open cover $\{p(L^z_i)\}$ are either empty or contractible, because they are homotopy equivalent to the corresponding intersections from the cover $\{L^z_{i}\}$ via the projection map $p$.
\item The nerve of $\{p(L^z_i)\}$ is the same as the nerve of $\{L^z_{i}\}$.
\item So, $\cup L^z_i$ is homotopy equivalent to $p(\cup L^z_i)$.
\item The nilpotent group $N$ or rank $r$ acts on the horosphere $H(x,\xi)\cong\mathbb R^{n-1}$, preserves each of the individual elements of the cover $\{p(L^z_i)\}$ and consequently acts on $p(\cup L^z_i)$ in a homotopically trivial way. 
\end{itemize}
Now, Proposition \ref{homologylemma} implies $H_{\geq n-1-r}(p(\cup L^z_i))=0$ and consequently $H_{\geq n-1-r}(\cup L^z_i)=0$. 
\end{proof}
Applied to big clumps: if $N_{\alpha}$ has rank $r$ and contains parabolics then $H_{\geq n-1-r}(Y^z_{\alpha})=0$. 

\subsection{\label{parabolicunfolding}Assembling parabolic clumps via unfolding space $^uL$} 
We show the following claim.
\begin{itemize}
\item 
Let $\varepsilon'=3^n(I_n!)\varepsilon$. For any $\varepsilon'<\mu$, if a cycle $c_p$ of degree $\geq n-1-r$ can be covered by a union of parabolic $\varepsilon$-patches $L=\cup L_i$ of rank $\geq r$ then it bounds inside the union of large $\varepsilon'$-patches $L^z=\cup L^z_{\sigma}$. 
\end{itemize}
Denote by $\Delta L$ the fattening of $L$ obtained from the cover $\{Y_\alpha\}$ (see subsection \ref{Delta}).
Since the cover $\{Y_\alpha\}$ is closed under intersections, its reduced nerve $\overline{\mathcal Y}$ can be described as the simplicial complex whose $k$-simplices are strictly decreasing chains $Y_{\alpha_0}\supset\dots\supset Y_{\alpha_k}$.
We let
$$
^{u}L:=\coprod Y^z_{\alpha_k}\times\{Y_{\alpha_0}\supset\dots\supset Y_{\alpha_k}\}/\sim
$$ 
be the quotient of the disjoint union of all the possible products $Y^z_{\alpha_k}\times\{Y_{\alpha_0}\supset\dots\supset Y_{\alpha_k}\}$ by identifications coming from the inclusions $Y^z_{\alpha_k}\times\{Y_{\beta_0}\supset\dots\supset Y_{\beta_l}\}\hookrightarrow Y^z_{\beta_l}\times\{Y_{\beta_0}\supset\dots\supset Y_{\beta_l}\}$ corresponding to subsets $\{\beta_0,\dots,\beta_l\}\subset\{\alpha_0,\dots,\alpha_k\}$. 
\begin{figure}
\centering
\includegraphics[scale=0.20]{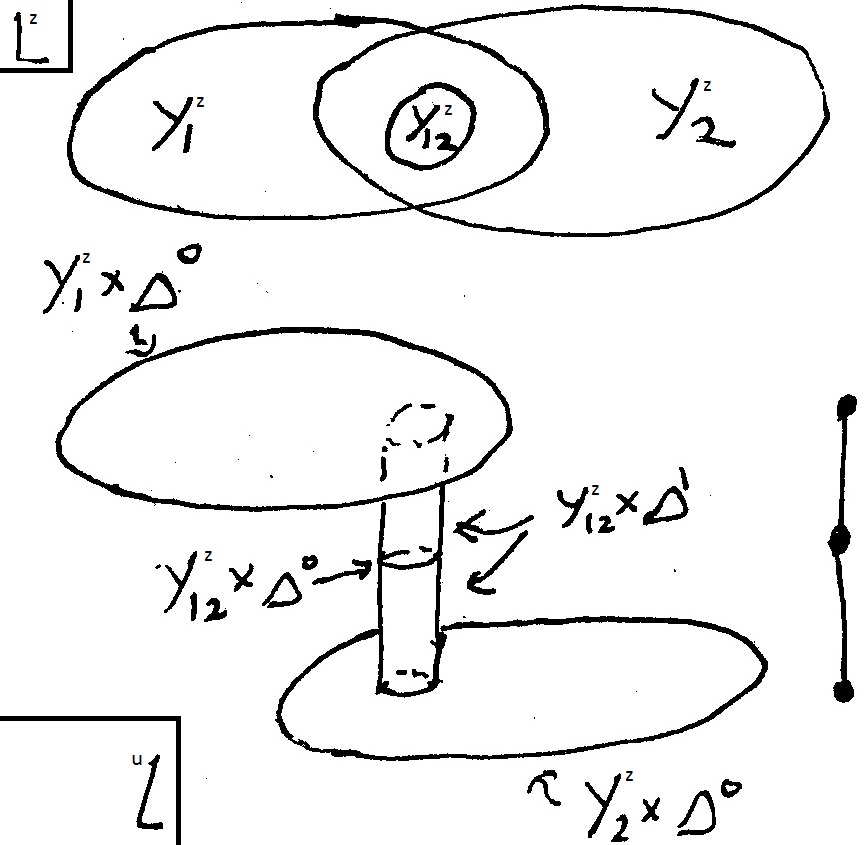}
\end{figure}

These spaces fit together into a commutative diagram
$$
\begin{array}{ccc}
&&\overline{\mathcal Y}\\
&&\uparrow\\
\Delta L&\ra&^{u}L\\
\downarrow&&\downarrow\\
L&\hookrightarrow&L^z.
\end{array}
$$ 
In this diagram
\begin{itemize}
\item
the left vertical arrow is a homotopy equivalence (by Proposition 4G.2 of \cite{hatcher}),
\item 
the top horizontal map is the composition of the map to the reduced thickening $\Delta L\ra\overline{\Delta}L$ with the inclusion $\overline{\Delta}L\hookrightarrow{^{u}}L$ obtained from the inclusions $Y_{\alpha}\hookrightarrow Y^z_{\alpha},$ and 
\item
the right vertical arrows are the obvious projections to $L^z$ and to the reduced nerve $\overline{\mathcal Y}$. 
\end{itemize} 
The point of this construction is that $^{u}L$ is a space where the elements of the cover $\{Y^z_\alpha\}$ are assembled according to the combinatorics of the reduced nerve $\overline{\mathcal Y}$ of the cover $\{Y_\alpha\}$. 
The diagram shows that to prove the claim it is enough to show that the homology of the space $^{u}L$ vanishes in dimensions $\geq  n-1-r$. We can compute the homology of ${^u}L$ using the map to the reduced nerve $^{u}L\ra\overline{\mathcal Y}$ via the corresponding spectral sequence

\begin{equation}
\label{mayervietoris}
\bigoplus_{Y_{\alpha_0}\supset\dots\supset Y_{\alpha_k}}H_j(Y^z_{\alpha_k}\times\{Y_{\alpha_0}\supset\dots\supset Y_{\alpha_k}\})
=\bigoplus_{Y_{\alpha_0}\supset\dots\supset Y_{\alpha_k}}H_j(Y^z_{\alpha_k})\implies H_{k+j}({^u}L),
\end{equation}
where the sum is over all the strictly decreasing chains $Y_{\alpha_0}\supset\dots\supset Y_{\alpha_k}$ representing simplices in the reduced nerve $\overline{\mathcal Y}$. To show that the term on the right is zero for $k+j\geq n-1-r$ we need to show that the terms on the left are all 
zero for $j\geq n-1-k-r$. In other words, we need to show that 
for any strictly decreasing chain $Y_{\alpha_0}\supset\dots\supset Y_{\alpha_k}$
\begin{equation}
H_{\geq n-1-(k+r)}(Y^z_{\alpha_k})=0.
\end{equation}
Since the patches $\{L_i\}$ are parabolic of rank $\geq r$ and each $N_{\alpha}$ virtually contains some $\Gamma_i$, all the $N_{\alpha}$ have rank $\geq r$ and contain parabolic elements. 
The growing ranks property (Corollary \ref{growingrank}) implies $N_{\alpha_k}$ has rank $\geq k+r$. Since it acts homotopically trivially on $Y^z_{\alpha_k}$ and contains parabolic elements, the claim follows from the local vanishing lemma (Lemma \ref{lvl}).

\begin{remark}
What we have done so far is sufficient to prove Theorem \ref{maintheorem2} in the case of bounded negative curvature ($-1<K<0$). In this case, for sufficiently small $\varepsilon$ all the $\varepsilon$-patches are parabolic. So, we pick $\varepsilon$ small enough, represent a homology class in $H_{n-2}(\partial\widetilde M)$ by a cycle $c\subset\partial\widetilde M\times[0,\infty)$ that is covered by parabolic $\varepsilon$-patches $\cup L_i$ and bound it inside the union of $\varepsilon'$-patches $\cup L^z_{\sigma}$ in the $\mu$-thin part $\widetilde M_{<\mu}\subset\partial\widetilde M\times[0,\infty)$, thus proving the theorem. \end{remark}

\begin{remark}
Notice that even if we are only interested in $H_{n-2}(\partial\widetilde M)$, we still need to deal with lower dimensional cycles and higher rank nilpotent groups in order to put things together via the Mayer-Vietoris assembly method. 
\end{remark}

\subsection{Reduction to the semisimple case}
\label{semisimplereduction1}
Suppose that all the $\Gamma_i$ have rank $\geq r$ and let 
\begin{eqnarray*}
S&=&\bigcup_{L_{\rho}\mbox{ semisimple}}L_{\rho},\\
P&=&\bigcup_{L_{\tau}\mbox{ parabolic}}L_\tau.
\end{eqnarray*}
Any patch $L_{\sigma}=L_{\rho}\cap L_{\tau}$ in the intersection $S\cap P$ is parabolic of rank $\geq r+1$,
so (by the previous subsection) any $\geq(n-1-r)$-dimensional cycle $c\subset S\cup P$ can be broken up as $c=c_s+c_p$ where $c_s\subset S\cup(S\cap P)^z$, $c_p\subset (S\cap P)^z\cup P$, and
$$
(S\cap P)^z:=\bigcup_{L_{\rho}\mbox{ semisimple, } L_{\tau}\mbox{ parabolic}}(L_\rho\cap L_{\tau})^z.
$$
\begin{lemma}
If $\Gamma_{\rho}$ is semisimple then $\Gamma_{\rho}=\left<Z_{\rho}\right>$ and for any $\sigma$ containing $\rho$ we have $Z_{\rho}<Z_{\sigma}$.
\end{lemma}
\begin{proof}
If $\Gamma_{\rho}$ is semisimple then it is abelian, so its generating set $S'_{\rho}$ lies in $Z_{\rho}\subset\Gamma_{\rho}$ implying that $\Gamma_{\rho}=\left<Z_\rho\right>$. 
The set $Z_\rho$ of semisimple elements lies in the center of $\Gamma_\sigma$ (by Lemma \ref{semisimple}) and consists of elements that have length $\leq 3^n$ in the generating set $S'_{\rho}\subset S'_{\sigma}$, so $Z_{\rho}\subset Z_{\sigma}$. 
\end{proof}
The lemma implies
\begin{itemize}
\item
For any patch $L_{\sigma}=L_{\rho}\cap L_{\tau}$ in the intersection $S\cap P$, the abelian group $\left<Z_{\sigma}\right>$ contains the rank $\geq r$ semisimple group $\left<Z_\rho\right>=\Gamma_{\rho}$ and also contains parabolic elements (by Corollary $9$). Thus $(Z_{\sigma},L^z_{\sigma})$ is a parabolic $\varepsilon'$-patch of rank $\geq r+1$. Consequently, the entire region $(S\cap P)^z\cup P$ can be covered by parabolic $\varepsilon'$-patches of rank $\geq r$ so we can bound $c_p$ inside a union of large $\varepsilon''$-patches, as long as $\varepsilon''=3^n(I_n!)\varepsilon'$ is less than $\mu$.
\item(Semisimple engulfing)
If $\Gamma_{\rho}$ is semisimple and $\rho\subset\sigma$ then $L^z_{\rho}\supset L^z_{\sigma}$. Hence the union $S^z=\cup_{\Gamma_i \mbox{ semisimple}} L_i^z$ contains $(S\cap P)^z$. It also (clearly) contains $S$, so it contains the cycle $c_s$. 
\end{itemize}

\subsection{Reduction from semisimple patches to minsets}
\label{semisimplereduction2} 
The next step is to pass from the union of sublevel sets $S^z$ to a union of minsets. We change notation so that it does not involve the superscript $z$ anymore. 
In this new notation, the original $S^z$ is a union $\cup_iU_i^0$ of $\varepsilon'$-patches $(Z_i,U_i^0)$. The $Z_i$ generate abelian groups $A_i=\left<Z_i\right>$ of semisimple isometries and we define
\begin{eqnarray*}
Z^k_i&=&(I_n!)^kZ_i,\\
U_i^k&=&\bigcap_{\gamma\in Z^k_i}\{x\mid d(x,\gamma x)<\varepsilon' (I_n!)^k\},\\
A_i^k&=&\left<Z_i^k\right>,\\
M^k_i&=&Min\left(A_i^k\right)
\end{eqnarray*}
which are nested via
$$
\begin{array}{ccccc}
U_i^0&\subset& U_i^1&\subset&\dots,\\
\cup&&\cup&&\\
M^0_i&\subset& M_i^1&\subset&\dots.
\end{array}
$$ 
The key additional point is that if $U^0_i\cap U^0_j\not=\emptyset$ then $M^{1}_i\cap M^{1}_j\not=\emptyset$:
\begin{proof}
If $U_i^0$ and $U_j^0$ intersect at some point $x$ then the Margulis lemma implies\footnote{if $\varepsilon'<$ Margulis constant} that $\left<A_i^0,A_j^0\right>$ is almost nilpotent so $\left<A_i^1,A_j^1\right>$ is actually nilpotent. Since it is generated by semisimple elements, Lemma \ref{semisimple} implies it is free abelian and consists entirely of semisimple elements. Its minset $Min\left(\left<A_i^1,A_j^1\right>\right)$ is nonempty and is contained in both $M_i^1$ and $M_j^1$. 
\end{proof}

Consequently, we have natural inclusions of nerves 
$$
Nerve\{U^0_i\}\subset Nerve\{M^{1}_i\}\subset Nerve\{U^{1}_i\}.
$$ 
Since, all the minsets and sublevel sets are convex the nerves are homotopy equivalent to the unions of the elements of the covers. This implies that for any cycle $c\in\cup U_i^0$ there is a cycle $c_m\in\cup M_i^1$ so that $c$ is homologous to $c_m$ inside the union of $(I_n!)\varepsilon'$-patches $\cup U_i^{1}$. 

What we have shown so far can be summarized as follows.
\begin{itemize}
\item
Suppose $c$ is a degree $\geq (n-1-r)$-cycle covered by $\varepsilon$-patches of rank $\geq r$. If $\varepsilon''<\mu$, then $c$ is homologous inside the $\mu$-thin part $\widetilde M_{<\mu}$ to a cycle $c_m$ in a union of minsets $\cup M^1_i$ of rank $\geq r$. 
\end{itemize}
It remains to show that the cycle $c_m$ bounds in the $\mu$-thin part. We will show that it bounds inside a union of larger minsets as long as this later union is still contained in the $\mu$-thin part (see Proposition \ref{almostabelianarrangements}). 

\subsection{Arrangements of minsets of small semisimple isometries}
\label{anothersemisimpleassemblylemma}
For any union of intersections of minsets $M=\cup_{\sigma}\cap_{i\in\sigma}Min(A_i)$, let $M^k=\cup_{\sigma}\cap_{i\in\sigma}Min(A^k_i)$ and  
$
A^k_{\sigma}=\left<A^k_i\mid i\in\sigma\right>.
$

\begin{lemma}If $\cap_{i\in\sigma}Min(A_i^k)\not=\emptyset$ and $(I_n!)^{k}\varepsilon'<\lambda_n$ (the Margulis constant), then $A^{k+1}_{\sigma}$ is an abelian group of semisimple isometries preserving each $Min(A_i^{k+1})$ with $i\in\sigma$. 
\end{lemma}

\begin{proof}
The Margulis lemma implies that if $\cap_{i\in\sigma}Min(A_i^k)\not=\emptyset$ then the group $A^k_{\sigma}$ is almost nilpotent and contains a nilpotent subgroup of index $\leq I_n$. Thus $A^{k+1}_{\sigma}$ is nilpotent. Since it is generated by semisimple isometries, it is actually abelian and consists exclusively of semisimple isometries (Lemma \ref{semisimple}). In particular, all of its elements commute with $A^{k+1}_i$ for $i\in\sigma$, so $A^{k+1}_{\sigma}$ preserves the minset $Min(A^{k+1}_i)$ whenever $i\in\sigma$. 
\end{proof}

\begin{proposition}
\label{almostabelianarrangements}
Suppose that $M$ is a union of intersections of minsets $\cup_{\sigma}\cap_{i\in\sigma} Min(A_i)$ with all $A^1_{\sigma}$ of rank $\geq r$, and $(I_n!)^{2^{n-2-r}}\varepsilon'<\mu$. Then the inclusion $M\hookrightarrow M^{2^{n-2-r}}$ is zero on $H_{\geq n-1-r}$.
\end{proposition}

Before proving this Proposition, we explain how it is used to finish the proof of Theorem \ref{maintheorem2}.

\begin{proof}[Proof of Theorem \ref{maintheorem2}]
Let $c$ be an $(n-2)$-cycle in $\partial\widetilde M\times[0,\infty)$. Shifting it to a sufficiently far translate $\partial\widetilde M\times[t,\infty)$, we may assume the cycle $c$ is covered by $\varepsilon$-patches of rank $\geq 1$, with $(I_n!)^{2^{n-3}+1}\varepsilon'<\mu$. Then $\varepsilon''<\mu$ so, by the bullet at the end of the previous subsection, $c$ is homologous inside $\widetilde M_{\mu}$ to a cycle $c_m$ lying in a union of minsets $\cup M_i^1$, with all the groups $A^1_i$ having rank at least $1$. Now, Proposition \ref{almostabelianarrangements} implies $c_m$ bounds inside the union of minsets $\cup_iM_i^{2^{n-3}+1}$ of the groups $A^{2^{n-3}+1}_i$. By the conditions on $\varepsilon$, this larger union of minsets is in the $\mu$-thin part $\widetilde M_{\mu}\subset\partial\widetilde M\times[0,\infty)$, so we are done.   
\end{proof}

\subsection{\label{semisimpleassembly} Proof of Proposition \ref{almostabelianarrangements}} 
Let 
$$
M_{>r}=\cup_{\mbox{rank}(A^1_{\sigma})>r}\cap_{i\in\sigma}Min(A^k_i)
$$ 
be the union of those intersections of minsets whose $A^1_{\sigma}$ have rank $>r$. Using excision, we write 
$$
H_*(M,M_{>r})\cong\oplus_{\alpha} H_*(Y_{\alpha},Y_{\alpha>r}),
$$ 
where each $Y_{\alpha}$ has the form 
$$
Y_{\alpha}:=\cup_{A^1_{\sigma}=_vN}\cap_{i\in\sigma}Min(A_i),
$$
where all the $A^1_{\sigma}$ are  virtually equal ($=_v$) to some free abelian group $N$ of rank $r$ and 
$$
Y_{\alpha>r}:=Y_{\alpha}\cap M_{>r}.
$$ 
\begin{lemma}
\label{semisimplevanish} 
For $k\geq 1$, 
$$
H_{\geq n-1-r}(Y^k_{\alpha})=0.
$$
\end{lemma}

\begin{proof}
After replacing $N$ by the intersection of the $A^1_{\sigma}$, we  may assume $N$ is contained in all the free abelian groups $A^1_{\sigma}$. Thus, $N$ commutes with all $A^k_i$ for $i\in\sigma$ and $k\geq 1$. By Lemma \ref{semisimplesplittinglemma}
\begin{itemize}
\item
the minset of $N$ splits as $Min(N)=C\times\mathbb R^r$ with $C$ convex of dimension $\leq n-r$, 
\item
the $A_i^k$-action on $C\times\mathbb R^r$ splits, 
\item 
if $p=p_1\circ p_{Min(N)}$ denotes the closest point projection to the minset followed by projection to the first factor, then 
$$
p_{Min(N)}(Min(A_i^k))=p(Min(A_i^k))\times\mathbb R^r
$$
with $p(Min(A_i^k))$ being the minset of the $A_i^k$-action on $C$, and
\item
$p_{Min(N)}(Min(A_i^k))=Min(N)\cap Min(A_i^k)$.
\end{itemize} 

The $A_i^k$-action on $C$ factors through the finite group $A_i^k/(N\cap A_i^k)$ acting by orientation preserving isometries. The minset of this finite group acting on $C$ is just its fixed point set, so either $p(Min(A_i^k))=C$ or $p(Min(A_i^k))$ has codimension $\geq 2$ in $C$. So, 
either $p(Y_{\alpha}^k)=C$ or $p(Y_{\alpha}^k)$ is a finite union of subsets of codimension $\geq 2$. Moreover, the last bullet shows the nerves of $\{Min(A_i^k)\}_{i\in\sigma, A^1_{\sigma}=_vN}$ and $\{p_{Min(N)}(Min(A_i^k))\}_{i\in\sigma, A^1_{\sigma}=_vN}$ are the same, so the inclusion 
$$
p(Y^k_{\alpha})\times\mathbb R^r=p_{Min(N)}(Y^k_{\alpha})\hookrightarrow Y^k_{\alpha}
$$ 
is a homotopy equivalence.
Since $C$ has dimension $\leq n-r$ we conclude
$$
0=H_{\geq n-r-1}(p(Y^k_{\alpha}))=H_{\geq n-r-1}(p_{Min(N)}(Y^k_{\alpha}))=H_{\geq n-1-r}(Y^k_{\alpha}).
$$
\end{proof}
Now, we induct. 

{\bf Base case $r=n-2$:} 
In this case $M_{>r}=\emptyset$, so $M$ is a disjoint union of the $Y_{\alpha}$. The previous lemma shows that $M\hookrightarrow M^1$ is zero on $H_{\geq n-1-r}$. 

{\bf Inductive step:}
Suppose we already know the Proposition for unions of intersections of minsets in which all $A^1_{\sigma}$ have rank $>r$.

Pick $d\geq n-1-r$ and denote by $'$ the exponent $^{2^{n-3-r}}$. The above lemma, together with induction (applied to $Y_{\alpha>r}$) shows that in the diagram 
$$
\begin{array}{ccccc}
&&H_d(Y_{\alpha},Y_{\alpha>r})&\ra&H_{d-1}(Y_{\alpha>r}),\\
&&\downarrow&&\downarrow 0\\
H_d(Y'_{\alpha})&\ra&H_d(Y'_{\alpha},Y'_{\alpha>r})&\ra& H_{d-1}(Y'_{\alpha>r})\\
||&&&&\\
0&&&&
\end{array}
$$
the middle vertical map is zero. 
Putting these maps together for all $\alpha$ and using excision, we get the diagram
$$
\begin{array}{ccc}
\oplus_{\alpha}H_{d}(Y_{\alpha},Y_{\alpha>r})&\cong&H_{d}(M,M_{>r})\\
0\downarrow&&\downarrow\\
\oplus_{\alpha}H_{d}(Y'_{\alpha},Y'_{\alpha>r})&\ra&H_{d}(M',M'_{>r})
\end{array}
$$ 
and we see that the right vertical map in this diagram is zero. 
Finally,
$$
M'_{>r}=\cup_{\mbox{rank}(A^1_{\sigma})>r}\cap_{i\in\sigma}Min(A_i')
$$
is a union of intersections of minsets with all groups $(A'_{\sigma})^1=A'^{+1}_{\sigma}$ having rank $>r$ so we can apply induction to it and the diagram
$$
\begin{array}{ccccc}
H_d(M''_{>r})&\ra&H_d(M'')&&\\
0\uparrow&&\uparrow&&\\
H_d(M'_{>r})&\ra&H_d(M')&\ra&H_d(M',M'_{>r})\\
&&\uparrow&&\uparrow 0\\
&&H_d(M)&\ra&H_d(M,M_{>r})
\end{array}
$$ 
shows that the map $H_d(M)\ra H_d(M'')$ is zero. Since $M''=M^{2^{n-2-r}}$, this proves the Proposition.

\begin{remark}
The inductive argument used in this subsection to assemble minsets cannot be used to assemble information about parabolic patches. The basic obstacle is that if $(S_i,L_i)$ is a parabolic patch of rank $r$, the larger $\Gamma_i$-invariant patch $(Z_i,L_i^z)$ has smaller rank (if the group $\Gamma_i$ is not abelian). This is the reason for doing things via the unfolding space in subsection \ref{parabolicunfolding}.
\end{remark}

\section{Covers and coefficients (Proof of Theorem \ref{maintheorem3})\label{coversandcoefficients}}
We do the same arguments as in the proof of Theorem \ref{maintheorem2} in the previous section, except now in the universal cover of the end $q:\widetilde E\ra E$. In other words, we use homology with coefficients in the $\pi_1E$-module $V:=\mathbb Z[\pi_1E]$. The Mayer-Vietoris assembly arguments work in exactly the same way for homology with coefficients. The reduction to minsets in subsections \ref{semisimplereduction1} and \ref{semisimplereduction2} can also be expressed in terms of homology and works the same way for homology with coefficients. The places where the extension isn't formal are the local vanishing lemmas (Proposition \ref{homologylemma} and Lemma \ref{semisimplevanish}). Notice that the proof of Lemma \ref{semisimplevanish} actually shows that $Y^k_{\alpha}$ is homotopy equivalent to an $(n-2-r)$-complex, so its homology with any coefficients vanishes in dimension $\geq n-1-r$. On the other hand, the analogue of Proposition \ref{homologylemma} requires more care. This is because a homotopically trivial action on a complex may not lift to an action of the universal cover of that complex. For example a $\mathbb Z/p$-action on the circle $S^1$ by rotations does not lift to a $\mathbb Z/p$-action on $\mathbb R$. However, in our situation the homotopically trivial action also preserves a cover by contractible sets with contractible intersections. In this case, the following analogue of Proposition \ref{homologylemma} shows that the action does lift to a homotopically trivial action on the universal (or more generally any regular) cover. 

\begin{proposition}
\label{covervanish}
Let $N$ be a rank $r$ nilpotent group acting on $\mathbb R^{n-1}$ by covering translations, and let $W\subset\mathbb R^{n-1}$ be an open submanifold that is preserved by the $N$-action. 
Suppose that $W=\cup_i W_i$ is a union of contractible open sets $W_i$, that each finite intersection of the $W_i$'s is either empty or contractible, and that the $N$-action preserves each $W_i$. Let $\hat W\ra W$ be a regular cover of with covering group $G$. 
Then, 
\begin{itemize}
\item 
the $N$-action on $W$ lifts to a homotopically trivial $N$-action on $\hat W$ commuting with $G$,
\item
if $r<n-1$ then 
$H_{\geq n-1-r}(\hat W)=0$, and 
\item
if $r=n-1$ then $\hat W=G\times W=G\times\mathbb R^{n-1}$.  
\end{itemize}
\end{proposition}
\begin{proof}
Pick specific lifts $\hat W_i$ of the $W_i$ to the cover $\hat W$ and express it as a union $\hat W=\cup_{g\in G}\cup_{i}g\hat W_i$ of $G$-translates of these lifts\footnote{So, each $\hat W_i$ is contractible and $\hat W_i\ra W_i$ is a homeomorphism.}. Define the $N$-action on $\hat W$ as follows. Let $\hat w\in \hat W$ be a point in the cover and $w\in W$ the corresponding point in the base. Then the point $\hat w$ is contained in some $g\hat W_i$ and we define $n\hat w$ to be the unique point in $g\hat W_i$ that lies over $nw\in W_i$. In other words, we define things so that the diagram 
$$
\begin{array}{ccc}
g\hat W_i&\stackrel{n}\ra&g\hat W_i\\
\downarrow&&\downarrow\\
W_i&\stackrel{n}\ra& W_i,
\end{array} 
$$
is commutative. It is easy to check that this is well defined (if $\hat z$ is also contained in $h\hat W_j$ and we used the diagram for $h\hat W_j$ instead we would get the same point $n\hat z\in\hat W$ because the two diagrams contain a common ``subdiagram'' corresponding to $g\hat W_i\cap h\hat W_j$) and that it defines a group action of $N$ on the cover $\hat W$. This action commutes with the covering space action of $G$ because $ngn^{-1}\in G$ and $ngn^{-1}\hat W_i=g\hat W_i$ implies $ngn^{-1}=g$ (the $G$-translates of $\hat W_i$ are disjoint because $W_i$ is contractible.) Finally, note that the $N$-action on $\hat W$ is homotopically trivial because it preserves each $g\hat W_i$ so the argument of Proposition \ref{homologylemma} applies. 
The spectral sequence of the bundle 
$\hat W\ra (\hat W\times\mathbb R^{n-1})/N\ra\mathbb R^{n-1}/N$
is
\begin{equation}
E^2_{i,j}=H_i(N;H_j(\hat W))\implies H_{i+j}((\hat W\times\mathbb R^{n-1})/N)=H_{i+j}(\hat W/N),
\end{equation}
and the $N$-module $H_j(\hat W)$ is trivial because the $N$-action on $\hat W$ is homotopically trivial. If $k$ is the largest integer for which $H_k(\hat W)\neq 0$, then the $E^2$-term $E^2_{r,k}=H_r(N;H_k(\hat W))\not=0$ never gets killed, implying $H_{r+k}(\hat W/N)\neq 0$. Since $W\subset \mathbb R^{n-1}$ is open, either
\begin{itemize}
\item $r+k<n-1$, so $r<n-1$ and  the homology $H_*(\hat W)$ vanishes for $*\geq n-1-r$, or 
\item
$r+k=n-1$. In this case $\hat W/N$ is a closed $(n-1)$-manifold, so $W/N$ is as well. This means $W=\mathbb R^{n-1}$, $k=0$ and $r=n-1$. Thus, the cover $\hat W\ra W$ is $G\times \mathbb R^{n-1}$. 
\end{itemize}
\end{proof}
We apply this Proposition in place of Proposition \ref{homologylemma} to show that for a big parabolic clump $Y^z_{\alpha}$ of rank $r\leq n-2$ and homology with coefficients in $V=\mathbb Z
\pi_1E$ we have 
$$
H_{\geq n-1-r}(Y^z_{\alpha};V)=H_{\geq n-1-r}(p(Y^z_{\alpha});V)=H_{\geq n-1-r}(q^{-1}p(Y^{z}_{\alpha}))=0.
$$ 
The rest of the argument from section \ref{mainsection} goes through, using homology with coefficients in $V$ instead of $\mathbb Z$ everywhere. The conclusion is $H_{\geq n-2}(\widetilde E)=0$. This is true for every component $E$ of the end $\partial\widetilde M\times[0,\infty)$ so we conclude $H_{\geq n-2}(\widetilde{\partial M})=0$. This proves equation (\ref{coeff}).


\section{Questions}
\subsection{Homotopy type of the boundary}
The (conceptually) simplest reason for the homology vanishing in dimension $\geq n-2$ would be that the boundary of the universal cover $\partial\widetilde M$ is actually homotopy equivalent to an $(n-3)$-complex. 
\begin{question}
Is $\partial\widetilde M$ homotopy equivalent to an $(n-3)$-complex?
\end{question}
\begin{remark}
In the $4$-dimensional case, this question is asking whether the boundary of the universal cover $\partial\widetilde M$ is homotopy equivalent to a graph. 
\end{remark}
A classical result of Wall, implies that the answer to this question is yes (at least in dimensions $n\not=5$) if, in addition to the homological vanishing established in this paper, a certain {\it cohomology} group vanishes.
\begin{theorem}[Wall]
Let $X$ be a complex and $\widetilde X$ its universal cover. Suppose that 
$$
H_{\geq n-2}(\widetilde X)=0,
$$
$$
H^{n-2}(X;V)=0
$$ 
for any $\pi_1X$-module $V$. If $n>2,n\not=5$ then $X$ is homotopy equivalent to an $(n-3)$-complex. 
\end{theorem}
\begin{remark}
This is a restatement of Theorems D and E of \cite{wallfiniteness}. 
\end{remark}
\begin{question}
Is $H^{n-2}(\partial\widetilde M;V)=0$ for every $\pi_1\partial\widetilde M$-module $V$?
\end{question}
One can do the argument of this paper in cohomology to show that for a {\it finite} union of small patches $L=L_{x_1}\cup\dots\cup L_{x_k}$ the corresponding unfolding space $^uL$ has vanishing $H^{n-2}(\cdot;V)$ and, consequently, this unfolding space $^uL$ is homotopy equivalent to an $(n-3)$-complex (when $n\not=5$). Moreover, in dimension $n=4$ all the relevant\footnote{Recall that if one of the nilpotent groups has rank $n-1=3$ then the component of the end $E$ is contractible, i.e. homotopy equivalent to a point, and we are done.} nilpotent groups have rank $\leq 2$ and hence are abelian, so the groups $\Gamma_x$ preserve the small patches, there is no need to introduce the large patches or the unfolding space, and the argument together with Wall's theorem shows that any finite union of small patches $L$ is homotopy equivalent to a graph. Consequently, in addition to the conclusion that in dimension $4$ the boundary $\partial M$ is aspherical (the result of Nguy$\tilde{\hat{\mathrm{e}}}$n Phan which led to the extensions in this paper) we also get 
\begin{corollary}
If $n=4$ then the fundamental group of each component of $\partial\widetilde M$ is locally free\footnote{A group is locally free if every finitely generated subgroup of it is free.}. 
\end{corollary}
In order to show that $H^{n-2}(\partial\widetilde M;V)=0$, one would need to do the argument of this paper in cohomology also for {\it infinite} unions of small patches. 
Let $\Lambda\subset\pi_1M$ be the subgroup of the fundamental group preserving a component $E$ of the end. For large enough $t$, we can cover the cross section $\partial E_t$ by a finite union of $\Lambda$-orbits of patches $\cup_{i=1}^k\Lambda L_{x_i}$ in the end, so that $\partial E_t\subset\cup_{i=1}^k\Lambda L_{x_i}\subset E$. 
One problem with proving cohomology vanishing for such infinite unions is that there may not be minimal nilpotent groups. Baumslag-Solitar type relations (in the sense of \cite{shalenbs}) seem to be the enemy here, so it might be easier to prove cohomology vanishing with some sort of {\it(noBS)-condition} on the fundamental group. The simplest one is 
\begin{itemize}
\item
For any pair of nontrivial elements $x,y\in\pi_1M$,
if $xy^sx^{-1}=y^t$ then $s=\pm t$.
\end{itemize}
Conversely, in the presence of such relations it may well be that cohomology does not vanish. 
\subsection{Homological collapse without finiteness restrictions}
The finiteness conditions on $M$ (the topological tameness condition and the Riemannian finite volume condition) only appear in the proof through the rank $(n-1)$ Lemma. Thus, one is led to wonder whether there is a result about homological collapse in the thin part of the universal cover that is valid without any such finiteness restrictions. First, we note that our arguments give such a result if we {\it assume} that all the relevant nilpotent groups have rank $\leq n-2$. We state it now.

Let $M$ be a complete $n$-manifold of bounded nonpositive curvature $-1\leq K\leq 0$, $\Gamma$ its fundamental group and $\widetilde M$ its universal cover. Let $\Gamma_x(\varepsilon):=\left<\gamma\in\Gamma\mid d(x,\gamma x)<\varepsilon\right>$ be the group generated by isometries that move the point $x$ by less than $\varepsilon$. The Margulis lemma says there is a constant $\lambda_n>0$, depending only on the dimension $n$, so that $\Gamma_x(\varepsilon)$ is almost nilpotent whenever $\varepsilon\leq\lambda_n$. 
For any such $\varepsilon$ we can filter the universal cover by ranks of these groups, with $\widetilde M_{<\varepsilon}^{(r)}:=\{x\in\widetilde M\mid \mbox{ rank }\Gamma_x(\varepsilon)\geq r\}$ being the set of points where the group generated by $\varepsilon$-small isometries has rank at least $r$. Before, we've called set $\widetilde M^{(1)}_{<\varepsilon}$ the ($\Gamma,\varepsilon$)-thin part (or simply the $\varepsilon$-thin part) of the universal cover and denoted it by $\widetilde M_{<\varepsilon}$.

\begin{theorem}
\label{collapsetheorem}
Suppose that
\begin{itemize}
\item[($\star$):]\footnote{Another way to state this condition is $\widetilde M^{(n-1)}_{<\lambda_n}=\emptyset$.} the almost nilpotent groups $\Gamma_x(\lambda_n)$ have rank $\leq n-2$ for all $x\in\widetilde M$.   
\end{itemize}
Then, there is a constant $C_n$ with the following property: 
\begin{itemize}
\item
if $\varepsilon<\lambda_n/C_n$ and $d\geq n-1-r$, then any $d$-cycle  in $\widetilde M^{(r)}_{<\varepsilon}$ bounds in $\widetilde M_{<C_n\varepsilon}$. 
\end{itemize}
Moreover, we can find such a $C_n$ depending only on the dimension $n$.
\end{theorem}
In particular, this says that 
if $n>2$ then any $(n-2)$-cycle in the $\varepsilon$-thin part of the universal cover bounds in the $C_n\varepsilon$-thin part. One 
is led to ask whether this result can be true if we allow some
rank $(n-1)$ nilpotent groups. This appears to be unknown even in the $n=3$ case. 
\begin{question}
Suppose $\widetilde M$ is a complete, simply connected, Riemannian $3$-manifold of bounded nonpositive curvature $-1\leq K\leq 0$. Let $\Gamma$ be a torsionfree, discrete group of isometries of $\widetilde M$. 
Is there a constant $C_3$ such that, for any $\varepsilon<\lambda_3/C_3$ a loop in the $(\Gamma,\varepsilon)$-thin part bounds in the $(\Gamma,C_3\varepsilon)$-thin part? If yes, can one pick a single constant $C_3$ that works for all such $3$-manifolds?
\end{question}

\subsection{Dimension five} In dimension $\geq 6$ it is easy to give examples where the boundaries are not aspherical. For instance, the cartesian product of three punctured tori is homeomorphic to a finite volume, complete, $6$-dimensional Riemannian manifold of bounded nonpositive curvature whose boundary is not aspherical. However, in dimension five there is still a chance that the boundary is aspherical. 
\begin{question}
Suppose that $M$ is a $5$-dimensional, complete, finite volume Riemannian manifold of bounded nonpositive curvature. If $M$ is tame, is $\partial M$ aspherical?
\end{question} 
\subsection{Lower bounds on homology}
In all known examples of complete, finite volume Riemannian manifolds of bounded nonpositive curvature, the reduced homology of $\partial\widetilde M$ is concentrated in a single dimension, and consequently the fundamental group $\pi_1M$ is a duality group. It would be interesting to either give examples where this is not the case (the manifold studied by Buyalo in \cite{buyalo} would be an example if one could show that the fundamental group of the boundary does not surject onto the fundamental group of the interior) or to find some mechanism that leads to vanishing of homology of $\partial\widetilde M$ in low dimensions and gives some concentration. 
In the arguments of this paper, three different scales are relevant.
\begin{itemize}
\item
(Patch scale): The patches are sublevel sets so they are convex. 
\item
(Clump scale): The largest scale on which one deals with a single nilpotent group.
\item
(Finite scale): This is where one has a finite union of clumps. 
\end{itemize}
If one finds constraints on the topology of a maximal clump beyond the local vanishing lemma, one can try to assemble them into a statement about 
the end via a Mayer-Vietoris argument.
\subsection{Abelian arrangements of minsets} A related question is to understand what conditions on an abelian arrangement of minsets give lower bounds on homology. 

\begin{question}
Suppose $\{A_i\}$ is a collection of abelian groups of semisimple isometries on $\widetilde M$, and that $A_i$ commutes with $A_j$ whenever $Min(A_i)$ intersects $Min(A_j)$. When is the union $\cup_i Min(A_i)$ homotopy equivalent to a wedge of spheres $\vee S^k$ concentrated in a single dimension? 
\end{question}
\subsection{Relation to other types of collapse}
In spirit, the homological collapse of this paper has many common themes with the metric collapse of $F$-structures \cite{cheegergromovcollapse1,cheegergromovcollapse2} and the Cheeger-Fukaya-Gromov theory of nilpotent structures \cite{cheegerfukayagromov}. It might be interesting to figure out if there is some direct relation, and whether the metric collapse theories give additional topological information about the end in the presence of nonpositive curvature. 
\subsection{Can the thick part be homotoped into the end?}
It would be very interesting if there is some topological obstruction arising from bounded nonpositive curvature that prevents the thick part from being homotoped into the end. For instance, one might hope that the patches on the end can be assembled into some sort of structure that cannot exist on the thick part for topological reasons. 
\begin{question}
Let $M$ be a complete, finite volume manifold of bounded nonpositive curvature. Suppose $M$ is tame. Is there a homotopy $h_t:M\ra M$ with $h_0=id_M$ and $h_1(M)\subset\partial M$?\footnote{Equivalently, does the inclusion $\partial M\hookrightarrow M$ have a homotopy section?}
\end{question}
\begin{remark}
For a closed higher genus surface $\Sigma$ the product $\Sigma\times\mathbb R$ has a complete, finite volume, negatively curved Riemannian metric \cite{nguyenphanfinite}, so the two-sided curvature bound is essential here. 
\end{remark}
\subsection{Infinite topological type}
One can ask for a version of Corollary \ref{indecomposable} if we drop tameness but retain the finite volume assumption.
\begin{question}
Let $n\geq 3$. Suppose $\pi$ is the fundamental group of a complete, finite volume Riemannian $n$-manifold of bounded nonpositive curvature. Is $\pi$ freely indecomposable?
\end{question}


\bibliographystyle{amsplain}
\bibliography{ends}

\end{document}